\documentclass[reqno, 12pt]{amsart}
\usepackage{a4wide}
\usepackage{amscd}
\input xy
\xyoption{all}

\newtheorem{thm}{Theorem}[section]

\newtheorem{prop}[thm]{Proposition}

\newtheorem{lm}[thm]{Lemma}

\newtheorem{tb}[thm]{Table}

\numberwithin{equation}{section}

\begin{document}

\title{Hypersurfaces in the noncompact Grassmann manifold $SU_{2,m}/S(U_2U_m)$}
\author{\textsc{J\"{u}rgen Berndt and Young Jin Suh}}

\address{Department of Mathematics \\ King's College London\\ Strand\\ London\\ WC2R 2LS \\ United Kingdom}
\address{Department of Mathematics \\ College of Natural Sciences\\ Kyungpook National University \\
Daegu 702-701\\ Republic of Korea}
\date{}

\begin{abstract}
The Riemannian symmetric space $SU_{2,m}/S(U_2U_m)$, $m \geq 2$, is both Hermitian
symmetric and quaternionic K\"{a}hler symmetric. Let $M$ be a
hypersurface in $SU_{2,m}/S(U_2U_m)$ and denote by $TM$ its tangent bundle.
The complex structure of $SU_{2,m}/S(U_2U_m)$ determines a maximal
complex subbundle ${\mathcal C}$ of $TM$, and the quaternionic
structure of $SU_{2,m}/S(U_2U_m)$ determines a maximal quaternionic
subbundle ${\mathcal Q}$ of $TM$. In this article we investigate
hypersurfaces in $SU_{2,m}/S(U_2U_m)$ for which
${\mathcal C}$ and ${\mathcal Q}$ are closely related to the shape of $M$.
\end{abstract}

\maketitle
\thispagestyle{empty}

\footnote[0]{The second author was supported by grant Proj. No. R17-2008-001-01001-0 from NRF
}


\section{Introduction}

Real hypersurfaces in complex space forms have been the source for a vast amount
of research activity in the last decades. Little is known though about real hypersurfaces
in other K\"{a}hler manifolds, which is of course due to the more complicated geometry
of other K\"{a}hler manifolds.
Of particular importance in this context are real hypersurfaces $M$ for which
the maximal complex subbundle ${\mathcal C}$ of the tangent bundle
$TM$ of $M$ is closely related to the shape of $M$.
The shape of $M$ is encoded in its second fundamental form $h$.
Let ${\mathcal C}^\perp = TM \ominus {\mathcal C}$ be
the orthogonal  of ${\mathcal C}$ in $TM$. The subbundle ${\mathcal C}^\perp$
has rank one and hence is always integrable. If the integral
manifolds are totally geodesic submanifolds of $M$, then $M$ is called a Hopf hypersurface.
For the special case of $S^{2m-1} \subset {\mathbb C}^m$ the
corresponding foliation is the well-known Hopf foliation.
It is not difficult to see that $M$ is a Hopf hypersurface if and only
if $h({\mathcal C},{\mathcal C}^\perp) = 0$, or equivalently,
if ${\mathcal C}$ is invariant under the shape operator $A$ of $M$.

In this paper we investigate Hopf hypersurfaces in the Hermitian
symmetric space $SU_{2,m}/S(U_2U_m)$, $m \geq 2$. This symmetric
space has rank two.
A major geometric difference between $SU_{2,m}/S(U_2U_m)$ and
its rank one partner, the complex hyperbolic space ${\mathbb C}H^m = SU_{1,m}/S(U_1U_m)$,
is the existence of geometrically inequivalent
tangent vectors. In ${\mathbb C}H^m$ all tangent
vectors are geometrically equivalent because of the two-point homogeneity
of ${\mathbb C}H^m$. On $SU_{2,m}/S(U_2U_m)$, however, there is a
one-parameter family of geometrically inequivalent tangent vectors,
and it therefore seems to be a good choice as ambient K\"{a}hler manifold for
investigating real hypersurfaces.

The Hermitian symmetric space $SU_{2,m}/S(U_2U_m)$ has
the remarkable feature that it is also a quaternionic K\"{a}hler symmetric space.
We denote
by $J$ the K\"{a}hler structure and by ${\mathfrak J}$ the
quaternionic K\"{a}hler structure on $SU_{2,m}/S(U_2U_m)$.
Let $M$
be a connected hypersurface in $SU_{2,m}/S(U_2U_m)$ and denote by
$TM$ the tangent bundle of $M$. The maximal complex subbundle of
$TM$ is defined by ${\mathcal C} = \{ X \in TM \mid JX \in TM\}$,
and the maximal quaternionic subbundle ${\mathcal Q}$ of $TM$ is
defined by ${\mathcal Q} = \{ X \in TM \mid {\mathfrak J}X \subset
TM\}$. The orthogonal complement ${\mathcal Q}^\perp = TM \ominus {\mathcal Q}$
is a subbundle of $TM$ with rank three.
In this article we deal with the classification problem of all Hopf hypersurfaces
in $SU_{2,m}/S(U_2U_m)$ for which $h({\mathcal Q},{\mathcal Q}^\perp) = 0$. This is equivalent to classifying
all real hypersurfaces in $SU_{2,m}/S(U_2U_m)$ for which both ${\mathcal C}$ and ${\mathcal Q}$
are invariant under the shape operator of $M$.

We first present a few hypersurfaces in $SU_{2,m}/S(U_2U_m)$ with these two properties.
We denote by $o \in SU_{2,m}/S(U_2U_m)$ the unique fixed point of the action of the isotropy group
$S(U_2U_m)$ on $SU_{2,m}/S(U_2U_m)$.

Firstly, consider the conic (or geodesic) compactification of
$SU_{2,m}/S(U_2U_m)$. The points in the boundary of this
compactification correspond to equivalence classes of asymptotic
geodesics in $SU_{2,m}/S(U_2U_m)$. Every geodesic in
$SU_{2,m}/S(U_2U_m)$ lies in a maximal flat, that is, a
two-dimensional Euclidean space embedded in $SU_{2,m}/S(U_2U_m)$ as
a totally geodesic submanifold. A geodesic in $SU_{2,m}/S(U_2U_m)$
is called singular if it lies in more than one maximal flat in
$SU_{2,m}/S(U_2U_m)$. A singular point at infinity is the
equivalence class of a singular geodesic in $SU_{2,m}/S(U_2U_m)$. Up
to isometry, there are exactly two singular points at infinity for
$SU_{2,m}/S(U_2U_m)$. The singular points at infinity correspond to
the geodesics in $SU_{2,m}/S(U_2U_m)$ which are determined by
nonzero tangent vectors $X$ with the property $JX \in {\mathfrak
J}X$ and $JX \perp {\mathfrak J}X$ respectively. Our first main result
is a geometric characterization of horospheres in $SU_{2,m}/S(U_2U_m)$ whose center
at infinity is singular.

\begin{thm} \label{horosphereintro}
Let $M$ be a horosphere in  $SU_{2,m}/S(U_2U_m)$, $m \geq 2$.
The following statements are equivalent:
\begin{itemize}
\item[(i)] the center of $M$ is a singular point at infinity,
\item[(ii)] $h({\mathcal C},{\mathcal C}^\perp) = 0$,
\item[(iii)] $h({\mathcal Q},{\mathcal Q}^\perp) = 0$.
\end{itemize}
\end{thm}

Secondly, consider the standard embedding of $SU_{2,m-1}$ in
$SU_{2,m}$. Then the orbit $SU_{2,m-1} \cdot o$ of $SU_{2,m-1}$
through $o$ is the Riemannian symmetric space $SU_{2,{m-1}}/S(U_2U_{m-1})$
embedded in $SU_{2,m}/S(U_2U_m)$ as a totally geodesic
submanifold. Every tube around $SU_{2,m-1}/S(U_2U_{m-1})$ in
$SU_{2,m}/S(U_2U_m)$ satisfies $h({\mathcal C},{\mathcal C}^\perp) = 0$
and $h({\mathcal Q},{\mathcal Q}^\perp) = 0$.

Finally, let $m$ be even, say $m = 2n$, and consider the standard
embedding of $Sp_{1,n}$ in $SU_{2,2n}$. Then the orbit $Sp_{1,n}
\cdot o$ of $Sp_{1,n}$ through $o$ is the quaternionic hyperbolic
space ${\mathbb H}H^n$ embedded in $SU_{2,2n}/S(U_2U_{2n})$
as a totally geodesic submanifold. Any tube around
${\mathbb H}H^n$ in $SU_{2,2n}/S(U_2U_{2n})$ satisfies
$h({\mathcal C},{\mathcal C}^\perp) = 0$ and $h({\mathcal Q},{\mathcal Q}^\perp) = 0$.

The second main result of this article states that with one possible exceptional case
there are no other such real hypersurfaces.

\begin{thm} \label{mainresult}
Let $M$ be a connected hypersurface in $SU_{2,m}/S(U_2U_m)$, $m \geq 2$.
Then $M$ satisfies $h({\mathcal C},{\mathcal C}^\perp) = 0$ and $h({\mathcal Q},{\mathcal Q}^\perp) = 0$ if and only if
$M$ is congruent to an open part of one of the following hypersurfaces:
\begin{itemize}
\item[(i)] a tube around a totally geodesic $SU_{2,m-1}/S(U_2U_{m-1})$
in $SU_{2,m}/S(U_2U_m)$;
\item[(ii)] a tube around a totally geodesic ${\mathbb H}H^n$ in
$SU_{2,2n}/S(U_2U_{2n})$, $m = 2n$;
\item [(iii)] a horosphere in $SU_{2,m}/S(U_2U_m)$
 whose center at infinity is singular;
\end{itemize}
or the following exceptional case holds:
\begin{itemize}
\item[(iv)] The normal bundle $\nu M$ of $M$ consists of singular tangent vectors
of type $JX \perp {\mathfrak J}X$.
Moreover, $M$ has at least four distinct principal curvatures,
three of which are given by
$$
\alpha = \sqrt{2}\ ,\ \gamma = 0\ ,\ \lambda =
\frac{1}{\sqrt{2}}
$$
with corresponding principal curvature spaces
$$
T_\alpha = ({\mathcal C} \cap {\mathcal Q})^\perp\ ,\ T_\gamma =
J{\mathcal Q}^\perp\ ,\ T_\lambda \subset  {\mathcal C} \cap {\mathcal Q} \cap J{\mathcal Q}.
$$
If $\mu$ is another (possibly nonconstant) principal
curvature function, then we have $T_\mu \subset {\mathcal C} \cap {\mathcal Q} \cap J{\mathcal Q}$,
$JT_\mu \subset T_\lambda$
and ${\mathfrak J}T_\mu \subset T_\lambda $.
\end{itemize}
\end{thm}

One of the main tools for the proof of Theorem \ref{mainresult} is the Codazzi equation,
which provides some useful relations between the principal curvatures of the hypersurface.
The exceptional case arises from particular values of possible principal curvatures
for which the Codazzi equation degenerates partially to the equation $0 = 0$ and therefore does
not provide sufficient information.
We conjecture that there are no real hypersurfaces in $SU_{2,m}/S(U_2U_m)$ whose principal curvatures
satisfy the conditions stated in Theorem \ref{mainresult} (iv).
It is remarkable that up to this possible exception all hypersurfaces satisfying
$h({\mathcal C},{\mathcal C}^\perp) = 0$ and $h({\mathcal Q},{\mathcal Q}^\perp) = 0$
are locally homogeneous.

The article is organised as follows. In Section \ref{horo}
we discuss some aspects of the geometry of horospheres in $SU_{2,m}/S(U_2U_m)$
and prove Theorem \ref{horosphereintro}. In Section \ref{curvature}
we present some basic material about the curvature of $SU_{2,m}/S(U_2U_m)$. In
Sections \ref{tube1} and \ref{tube2} we investigate the geometry of the tubes
around the totally geodesic submanifold $SU_{2,m-1}/S(U_2U_m)$ in $SU_{2,m}/S(U_2U_m)$
and around the totally geodesic submanifold ${\mathbb H}H^n = Sp_{1,n}/Sp_1Sp_n$ in $SU_{2,2n}/S(U_2U_{2n})$.
We show in particular that $h({\mathcal C},{\mathcal C}^\perp) = 0$ and $h({\mathcal Q},{\mathcal Q}^\perp) = 0$
holds for every tube around any of these two totally geodesic submanifolds.
In Section \ref{proof} we present the proof of Theorem \ref{mainresult}. A key step is
Proposition \ref{key} where we show that the normal bundle of a hypersurface in $SU_{2,m}/S(U_2U_m)$
with $h({\mathcal C},{\mathcal C}^\perp) = 0$ and $h({\mathcal Q},{\mathcal Q}^\perp) = 0$ consists
of singular tangent vectors.

We finally mention that the corresponding classification for the compact
Riemannian symmetric space $SU_{2+m}/S(U_2U_m)$
was obtained in \cite{BS}. However, the noncompactness
of $SU_{2,m}/S(U_2U_m)$ leads to problems which require different
methods.

\section{Horospheres in $SU_{2,m}/S(U_2U_m)$} \label{horo}

The Riemannian symmetric space $SU_{2,m}/S(U_2U_m)$
is a connected, simply connected, irreducible
Riemannian symmetric space of noncompact type and with rank two. Let $G =
SU_{2,m}$ and $K = S(U_2U_m)$, and denote by ${\mathfrak g}$ and
${\mathfrak k}$ the corresponding Lie algebra. Let $B$ be the
Killing form of ${\mathfrak g}$ and denote by ${\mathfrak p}$ the
orthogonal complement of ${\mathfrak k}$ in ${\mathfrak g}$ with
respect to $B$. The resulting decomposition ${\mathfrak g} =
{\mathfrak k} \oplus {\mathfrak p}$ is a Cartan decomposition of
${\mathfrak g}$. The Cartan involution $\theta \in {\rm Aut}({\mathfrak g})$ on
${\mathfrak s}{\mathfrak u}_{2,m}$ is given by $\theta(A) = I_{2,m} A I_{2,m}$, where
$ I_{2,m} = \begin{pmatrix} -I_2 & 0_{2,m} \\ 0_{m,2} & I_m \end{pmatrix}$, and
$I_2$ and $I_m$ is the identity $(2 \times 2)$-matrix and $(m \times m)$-matrix
respectively. Then $\langle X , Y \rangle =
-B(X,\theta Y)$ is a positive definite ${\rm Ad}(K)$-invariant inner
product on ${\mathfrak g}$. Its restriction to ${\mathfrak p}$
induces a Riemannian metric $g$ on $SU_{2,m}/S(U_2U_m)$, which is also
known as the Killing metric on $SU_{2,m}/S(U_2U_m)$.
Throughout this paper we consider $SU_{2,m}/S(U_2U_m)$
together with this particular Riemannian metric $g$.

The Lie algebra ${\mathfrak k}$ decomposes orthogonally into
${\mathfrak k}  = {\mathfrak
s}{\mathfrak u}_2 \oplus {\mathfrak s}{\mathfrak u}_m \oplus
{\mathfrak u}_1$, where ${\mathfrak u}_1$ is the one-dimensional center of
${\mathfrak k}$. The adjoint action of ${\mathfrak s}{\mathfrak
u}_2$ on ${\mathfrak p}$ induces the quaternionic K\"{a}hler structure
${\mathfrak J}$ on $SU_{2,m}/S(U_2U_m)$, and the adjoint
action of
$$
Z = \begin{pmatrix} \frac{mi}{m+2}I_2 & 0_{2,m} \\ 0_{m,2} & \frac{-2i}{m+2}I_m \end{pmatrix} \in {\mathfrak u}_1
$$
induces the K\"{a}hler structure $J$ on $SU_{2,m}/S(U_2U_m)$.
By construction, $J$ commutes
with each almost Hermitian structure $J_1$ in ${\mathfrak J}$.
Recall that a canonical local basis $J_1,J_2,J_3$ of a
quaternionic K\"{a}hler structure ${\mathfrak J}$ consists of three
almost Hermitian structures $J_1,J_2,J_3$ in ${\mathfrak J}$ such
that $J_\nu J_{\nu+1} = J_{\nu + 2} = - J_{\nu+1} J_\nu$, where the
index $\nu$ is to be taken modulo $3$. The tensor
field $JJ_\nu$, which is locally defined on $SU_{2,m}/S(U_2U_m)$,
is selfadjoint and satisfies $(JJ_\nu)^2 = I$ and ${\rm
tr}(JJ_\nu) = 0$, where $I$ is the identity transformation.
For a nonzero tangent vector $X$ we define
${\mathbb R}X = \{\lambda X \mid \lambda \in {\mathbb R}\}$,
${\mathbb C}X = {\mathbb R}X \oplus {\mathbb R}JX$, and
${\mathbb H}X = {\mathbb R}X \oplus {\mathfrak J}X$.

We identify the tangent space
$T_oSU_{2,m}/S(U_2U_m)$ of $SU_{2,m}/S(U_2U_m)$ at $o$
with ${\mathfrak p}$ in the usual way. Let ${\mathfrak a}$ be a
maximal abelian subspace of ${\mathfrak p}$. Since $SU_{2,m}/S(U_2U_m)$
has rank two, the dimension of any such subspace is two.
Every nonzero tangent vector $X \in T_oSU_{2,m}/S(U_2U_m)
\cong {\mathfrak p}$ is contained in some maximal abelian subspace
of ${\mathfrak p}$. Generically this subspace is uniquely determined
by $X$, in which case $X$ is called regular. If there exists more
than one maximal abelian subspaces of ${\mathfrak p}$ containing
$X$, then $X$ is called singular. There is a simple and useful
characterization of the singular tangent vectors: A nonzero tangent
vector $X \in {\mathfrak p}$ is singular if and only if $JX \in
{\mathfrak J}X$ or $JX \perp {\mathfrak J}X$.

Let ${\mathfrak a}^*$ be the dual vector space of ${\mathfrak a}$.
For each $\lambda \in {\mathfrak a}^*$
we define ${\mathfrak g}_\lambda = \{ X \in {\mathfrak g} \mid {\rm
ad}(H)X = \lambda(H)X\ {\rm for\ all}\ H \in {\mathfrak a}\}$. If
$\lambda \neq 0$ and ${\mathfrak g}_\lambda \neq \{0\}$, then
$\lambda$ is called a restricted root and ${\mathfrak g}_\lambda$ is
called a restricted root space. Let $\Sigma \subset {\mathfrak a}^*$
be the set of restricted roots. For each $\lambda \in \Sigma$ we
define $H_\lambda \in {\mathfrak a}$ by $\lambda(H) = \langle
H_\lambda , H \rangle$ for all $H \in {\mathfrak a}$. Since
${\mathfrak a}$ is abelian we get a restricted root space
decomposition ${\mathfrak g} = {\mathfrak g}_0 \oplus \left(
\bigoplus_{\lambda \in \Sigma} {\mathfrak g}_\lambda \right)$, where
${\mathfrak g}_0 = {\mathfrak k}_0 \oplus {\mathfrak a}$ and
${\mathfrak k}_0 \cong {\mathfrak u}_{m-2} \oplus {\mathfrak u}_1$ is
the centralizer of ${\mathfrak a}$ in ${\mathfrak k}$. The
corresponding restricted root system is of type $(BC)_2$. We choose
a set $\Lambda = \{\alpha_1,\alpha_2\}$ of simple roots of $\Sigma$
such that $\alpha_1$ is the longer root of the two simple roots, and
denote by $\Sigma^+$ the resulting set of positive restricted roots.
If we write, as usual, $\alpha_1 = \epsilon_1 - \epsilon_2$ and
$\alpha_2 = \epsilon_2$, the positive restricted roots are $\alpha_1
= \epsilon_1 - \epsilon_2$, $\alpha_2 = \epsilon_2$, $\alpha_1 +
\alpha_2 = \epsilon_1$, $2\alpha_2 = 2\epsilon_2$, $\alpha_1 +
2\alpha_2 = \epsilon_1 + \epsilon_2$ and $2\alpha_1 + 2\alpha_2 =
2\epsilon_1$. The multiplicities of the restricted roots $2\alpha_2$
and $2\alpha_1 + 2\alpha_2$ are $1$, the multiplicities of the
restricted roots $\alpha_1$ and $\alpha_1 + 2\alpha_2$ are $2$, and
the multiplicities of $\alpha_2$ and $\alpha_1 + \alpha_2$ are
$2m-4$, respectively. We denote by $\bar{C}^+(\Lambda)$ the closed
positive Weyl chamber in ${\mathfrak a}$ which is determined by
$\Lambda$. Note that $\bar{C}^+(\Lambda)$ is the closed cone in
${\mathfrak a}$ bounded by the half-lines spanned by
$H_{\alpha_1+\alpha_2}$ and $H_{\alpha_1 + 2\alpha_2}$.

We define a nilpotent subalgebra ${\mathfrak n}$ of ${\mathfrak g}$
by ${\mathfrak n} = \bigoplus_{\lambda \in \Sigma^+} {\mathfrak
g}_\lambda$. Then ${\mathfrak g} = {\mathfrak k} \oplus {\mathfrak
a} \oplus {\mathfrak n}$ is an Iwasawa decomposition of ${\mathfrak
g}$. The subalgebra ${\mathfrak s} = {\mathfrak a} \oplus {\mathfrak
n}$ of ${\mathfrak g}$ is solvable, and the corresponding connected
subgroup $S$ of $G$ with Lie algebra ${\mathfrak s}$ is solvable,
simply connected, and acts simply transitively on $SU_{2,m}/S(U_2U_m)$.
Let $H \in {\mathfrak a}$ be a unit vector. Then
${\mathfrak s}_H = {\mathfrak s} \ominus {\mathbb R}H$ is a
subalgebra of ${\mathfrak s}$ with codimension one. The connected
subgroup $S_H$ of $S$ with Lie algebra ${\mathfrak s}_H$ acts on
$SU_{2,m}/S(U_2U_m)$ with cohomogeneity one. If $H \in
\bar{C}^+(\Lambda)$, then the orbits of the action are the
horospheres in $SU_{2,m}/S(U_2U_m)$ which are determined by
the geodesic $\gamma_H$ with $\gamma_H(0) = o$ and $\dot\gamma_H(0)
= H$. We recall from \cite{BT} that the shape operator $A_H$ of the
horosphere $S_H \cdot o$, the orbit of $S_H$ through $o$, with respect to the unit
normal vector $H$ is the adjoint transformation $A_H = {\rm ad}(H)$
restricted to ${{\mathfrak s}_H}$.

Recall that the conic compactification of $SU_{2,m}/S(U_2U_m)$
is given by adding to the symmetric space $SU_{2,m}/S(U_2U_m)$ the equivalence
classes of its asymptotic geodesics, and then
equipping the resulting set with the cone topology.
If we denote by $SU_{2,m}/S(U_2U_m)(\infty)$
the boundary of $SU_{2,m}/S(U_2U_m)$ with
respect to the conic compactification, then the equivalence class
$[\gamma_H] \in SU_{2,m}/S(U_2U_m)(\infty)$ of all geodesics
in $SU_{2,m}/S(U_2U_m)$ which are asymptotic to $\gamma_H$ can
be viewed as the center of the horospheres given by the
$S_H$-action. A point in $SU_{2,m}/S(U_2U_m)(\infty)$
is called singular, if the geodesics in the corresponding
equivalence class are all singular. It is worthwhile to mention that
if the tangent vector to a geodesic in $SU_{2,m}/S(U_2U_m)$ is
singular at one point, than it is singular at every point. Thus it
makes sense to talk about singular geodesics. Moreover, if two
geodesics are asymptotic and one of them is singular, then the other
one must be singular as well. Therefore we can say that a point at
infinity is singular if the corresponding equivalence class of
asymptotic geodesics consists of singular geodesics. For details
about the conic compactification and points at infinity we refer to
\cite{E}.

We now proceed with some explicit calculations. We denote by $M_{k_1,k_2}({\mathbb C})$ the real
vector space of all $(k_1 \times k_2)$-matrices with complex coefficients, and by $0_{k_1,k_2}$ the
$(k_1 \times k_2)$-matrix with all coefficients equal to $0$.
For $a = (a_1,a_2) \in {\mathbb R}^2$ we put $\Delta_{2,2}(a) = \begin{pmatrix}
a_1 & 0 \\ 0 & a_2 \end{pmatrix}$. Then we have
\begin{eqnarray*}
{\mathfrak g} & = & \left\{ \left.
\begin{pmatrix} A & C \\ C^* & B\end{pmatrix}
\right|  A \in {\mathfrak u}_2,\ B \in {\mathfrak u}_m,\ {\rm tr}(A) + {\rm tr}(B) = 0,\ C \in M_{2,m}({\mathbb C})  \right\}, \\
{\mathfrak k} & = & \left\{ \left. \begin{pmatrix} A & 0_{2,m} \\
0_{m,2} & B\end{pmatrix}
\right|  A \in {\mathfrak u}_2,\ B \in {\mathfrak u}_m,\ {\rm tr}(A) + {\rm tr}(B) = 0 \right\}, \\
{\mathfrak p} & = & \left\{ \left. \begin{pmatrix} 0_{2,2} & C \\
C^* & 0_{m,m} \end{pmatrix}
\right|  C \in M_{2,m}({\mathbb C})  \right\}, \\
{\mathfrak a} & = & \left\{ \left. \begin{pmatrix} 0_{2,2} &
\Delta_{2,2}(a) & 0_{2,m-2} \\ \Delta_{2,2}(a) & 0_{2,2} & 0_{2,m-2} \\
0_{m-2,2} & 0_{m-2,2} & 0_{m-2,m-2} \end{pmatrix}
\right|  a \in {\mathbb R}^2  \right\}.
\end{eqnarray*}
The two vectors
$$
e_1 = \begin{pmatrix} 0_{2,2} &
\Delta_{2,2}(1,0) & 0_{2,m-2} \\ \Delta_{2,2}(1,0) & 0_{2,2} & 0_{2,m-2} \\
0_{m-2,2} & 0_{m-2,2} & 0_{m-2,m-2} \end{pmatrix}\ ,\
e_2 =  \begin{pmatrix} 0_{2,2} &
\Delta_{2,2}(0,1) & 0_{2,m-2} \\ \Delta_{2,2}(0,1) & 0_{2,2} & 0_{2,m-2} \\
0_{m-2,2} & 0_{m-2,2} & 0_{m-2,m-2} \end{pmatrix}.
$$
form a basis for ${\mathfrak a}$. We denote by $\epsilon_1,\epsilon_2 \in {\mathfrak a}^*$ the
dual vectors of $e_1,e_2$. Then the root system
$\Sigma$, the positive roots $\Sigma^+$, and the simple roots
$\Lambda = \{\alpha_1,\alpha_2\}$ are given by
$\Sigma = \{ \pm \epsilon_1 \pm \epsilon_2,\pm \epsilon_1,\pm \epsilon_2,\pm 2\epsilon_1,\pm2\epsilon_2 \}$,
$\Sigma^+ = \{ \epsilon_1 + \epsilon_2,\epsilon_1 - \epsilon_2,\epsilon_1,\epsilon_2,2\epsilon_1,2\epsilon_2 \}$,
$\alpha_1 = \epsilon_1 - \epsilon_2$, $\alpha_2 = \epsilon_2$.
For each $\lambda \in \Sigma$ we define the corresponding restricted root space ${\mathfrak p}_\lambda$ in
${\mathfrak p}$ by
$ {\mathfrak p}_\lambda = ({\mathfrak g}_\lambda \oplus {\mathfrak g}_{-\lambda}) \cap {\mathfrak p}$.
Then we have
${\mathfrak p}_0 = {\mathfrak a}$ and

\begin{eqnarray*}
{\mathfrak p}_{\epsilon_1} & = & \left\{ \left. \begin{pmatrix}
0 & 0 & 0 & 0 & v_1 & \cdots & v_{m-2} \\
0 & 0 & 0 & 0 & 0 & \cdots & 0 \\
0 & 0 & 0 & 0 & 0 & \cdots & 0 \\
0 & 0 & 0 & 0 & 0 & \cdots & 0 \\
\bar{v}_1 & 0 & 0 & 0 & 0 & \cdots & 0 \\
\vdots & \vdots & \vdots & \vdots & \vdots & \ddots & \vdots\\
\bar{v}_{m-2} & 0 & 0 & 0 & 0 & \cdots & 0 \\
\end{pmatrix}
\right|  v_1,\ldots,v_{m-2} \in {\mathbb C}  \right\} \cong {\mathbb C}^{m-2},
\end{eqnarray*}
\begin{eqnarray*}
{\mathfrak p}_{\epsilon_2} & = & \left\{ \left. \begin{pmatrix}
0 & 0 & 0 & 0 & 0 & \cdots & 0 \\
0 & 0 & 0 & 0 & v_1 & \cdots & v_{m-2}\\
0 & 0 & 0 & 0 & 0 & \cdots & 0 \\
0 & 0 & 0 & 0 & 0 & \cdots & 0 \\
0 & \bar{v}_1 & 0 & 0 & 0 & \cdots & 0 \\
\vdots & \vdots & \vdots & \vdots & \vdots & \ddots & \vdots\\
0 & \bar{v}_{m-2} & 0 & 0 & 0 & \cdots & 0 \\
\end{pmatrix}
\right|  v_1,\ldots,v_{m-2} \in {\mathbb C}  \right\} \cong {\mathbb C}^{m-2},
\end{eqnarray*}
\begin{eqnarray*}
{\mathfrak p}_{2\epsilon_1} & = & \left\{ \left. \begin{pmatrix}
0 & 0 & ix & 0 & 0 & \cdots & 0 \\
0 & 0 & 0 & 0 & 0 & \cdots & 0 \\
-ix & 0 & 0 & 0 & 0 & \cdots & 0 \\
0 & 0 & 0 & 0 & 0 & \cdots & 0 \\
0 & 0 & 0 & 0 & 0 & \cdots & 0 \\
\vdots & \vdots & \vdots & \vdots & \vdots & \ddots & \vdots\\
0 & 0 & 0 & 0 & 0 & \cdots & 0 \\
\end{pmatrix}
\right| x \in {\mathbb R}  \right\} \cong {\mathbb R},
\end{eqnarray*}
\begin{eqnarray*}
{\mathfrak p}_{2\epsilon_2} & = & \left\{ \left. \begin{pmatrix}
0 & 0 & 0 & 0 & 0 & \cdots & 0 \\
0 & 0 & 0 & ix & 0 & \cdots & 0\\
0 & 0 & 0 & 0 & 0 & \cdots & 0 \\
0 & -ix & 0 & 0 & 0 & \cdots & 0 \\
0 & 0 & 0 & 0 & 0 & \cdots & 0 \\
\vdots & \vdots & \vdots & \vdots & \vdots & \ddots & \vdots\\
0 & 0 & 0 & 0 & 0 & \cdots & 0 \\
\end{pmatrix}
\right|  x \in {\mathbb R}  \right\} \cong {\mathbb R},
\end{eqnarray*}
\begin{eqnarray*}
{\mathfrak p}_{\epsilon_1-\epsilon_2} & = & \left\{ \left. \begin{pmatrix}
0 & 0 & 0 & z & 0 & \cdots & 0 \\
0 & 0 & \bar{z} & 0 & 0 & \cdots & 0 \\
0 & z & 0 & 0 & 0 & \cdots & 0 \\
\bar{z} & 0 & 0 & 0 & 0 & \cdots & 0 \\
0 & 0 & 0 & 0 & 0 & \cdots & 0 \\
\vdots & \vdots & \vdots & \vdots & \vdots & \ddots & \vdots\\
0 & 0 & 0 & 0 & 0 & \cdots & 0 \\
\end{pmatrix}
\right| z \in {\mathbb C}  \right\} \cong {\mathbb C},
\end{eqnarray*}
\begin{eqnarray*}
{\mathfrak p}_{\epsilon_1+\epsilon_2} & = & \left\{ \left. \begin{pmatrix}
0 & 0 & 0 & -z & 0 & \cdots & 0 \\
0 & 0 & \bar{z} & 0 & 0 & \cdots & 0\\
0 & z & 0 & 0 & 0 & \cdots & 0 \\
-\bar{z} & 0 & 0 & 0 & 0 & \cdots & 0 \\
0 & 0 & 0 & 0 & 0 & \cdots & 0 \\
\vdots & \vdots & \vdots & \vdots & \vdots & \ddots & \vdots\\
0 & 0 & 0 & 0 & 0 & \cdots & 0 \\
\end{pmatrix}
\right|  z \in {\mathbb C}  \right\} \cong {\mathbb C}.
\end{eqnarray*}

For $t \in [0,\pi/4]$ we define
$$ H_t = \cos(t) e_1 + \sin(t) e_2 \in {\mathfrak a} $$
and denote by $M_t$ the horosphere which coincides with the orbit
$S_{H_t} \cdot o$. Every horosphere in $SU_{2,m}/S(U_2U_m)$ is
isometrically congruent to $M_t$ for some $t \in [0,\pi/4]$, and two
horospheres $M_{t_1}$ and $M_{t_2}$ are isometrically congruent if
and only if $t_1 = t_2$. The principal curvatures of $M_t$ with
respect to the unit normal vector $H_t$ are $0$ and $\lambda(H_t)$,
$\lambda \in \Sigma^+$, and ${\mathfrak a} \ominus {\mathbb R}H_t$
and ${\mathfrak p}_\lambda$ consists of corresponding principal
curvature vectors.

\begin{tb} \label{table1} The principal curvatures and corresponding eigenspaces and multiplicites of the
horosphere determined by $H_t = \cos(t) e_1 + \sin(t) e_2 \in {\mathfrak a}$ are given by
\smallskip
\begin{center}
\begin{tabular}{|l|l|l|}
\hline
\mbox{principal curvature} & \mbox{eigenspace}  & \mbox{multiplicity} \\
\hline
$0$ & ${\mathfrak a} \ominus {\mathbb R}H_t$ & $1$ \\
$2 \cos(t)$ & ${\mathfrak p}_{2\epsilon_1}$ & $1$ \\
$2 \sin(t)$ & ${\mathfrak p}_{2\epsilon_2}$ & $1$ \\
$\cos(t) - \sin(t)$ & ${\mathfrak p}_{\epsilon_1 - \epsilon_2}$ & $2$ \\
$\cos(t) + \sin(t)$ & ${\mathfrak p}_{\epsilon_1 + \epsilon_2}$ & $2$ \\
$\cos(t)$ & ${\mathfrak p}_{\epsilon_1}$ & $2m-4$ \\
$\sin(t)$ & ${\mathfrak p}_{\epsilon_2}$ & $2m-4$ \\
\hline\end{tabular}
\end{center}
\end{tb}

Thus the number of distinct principal curvatures is $7$ for $m > 2$ and $5$ for $m=2$
unless $t \in \{0,\arctan(\frac{1}{2}),\frac{\pi}{4}\}$. In these three cases we get the following table:

\begin{tb} \label{table2} The principal curvatures and corresponding eigenspaces and multiplicites of the
horosphere determined by $H_t = \cos(t) e_1 + \sin(t) e_2 \in {\mathfrak a}$ with $t \in \{0,\arctan(\frac{1}{2}),\frac{\pi}{4}\}$
are given by
\smallskip
\begin{center}
\begin{tabular}{|l|l|l|l|}
\hline
$t$ & \mbox{principal curvature} & \mbox{eigenspace}  & \mbox{multiplicity}\\
\hline
$0$ & $0$ & ${\mathbb R}e_2 \oplus {\mathfrak p}_{\epsilon_2} \oplus {\mathfrak p}_{2\epsilon_2}$ & $2m-2$ \\
& $1$ & ${\mathfrak p}_{\epsilon_1} \oplus {\mathfrak p}_{\epsilon_1 - \epsilon_2}
\oplus {\mathfrak p}_{\epsilon_1 + \epsilon_2}$ & $2m$\\
& $2$ & ${\mathfrak p}_{2\epsilon_1}$ & $1$ \\
\hline
$\arctan(\frac{1}{2})$ & $0$ & ${\mathbb R}(e_1-2e_2)$ & $1$ \\
& $1/\sqrt{5}$ & ${\mathfrak p}_{\epsilon_2} \oplus {\mathfrak p}_{\epsilon_1 - \epsilon_2}$ & $2m-2$\\
& $2/\sqrt{5}$ & ${\mathfrak p}_{\epsilon_1} \oplus {\mathfrak p}_{2\epsilon_2}$ & $2m-3$\\
& $3/\sqrt{5}$ & ${\mathfrak p}_{\epsilon_1 + \epsilon_2}$ & $2$\\
& $4/\sqrt{5}$ & ${\mathfrak p}_{2\epsilon_1}$ & $1$\\
\hline
$\frac{\pi}{4}$ & $0$ & ${\mathbb R}(e_1 -e_2) \oplus {\mathfrak p}_{\epsilon_1 - \epsilon_2}$ & $3$ \\
& $1/\sqrt{2}$ & ${\mathfrak p}_{\epsilon_1} \oplus {\mathfrak p}_{\epsilon_2}$ & $4m-8$\\
& $\sqrt{2}$ & ${\mathfrak p}_{2\epsilon_1} \oplus {\mathfrak p}_{2\epsilon_2} \oplus
{\mathfrak p}_{\epsilon_1 + \epsilon_2}$ & $4$ \\
\hline
\end{tabular}
\end{center}
\end{tb}

We now investigate the maximal complex subbundle ${\mathcal C}_t$ of $TM_t$. We recall that the complex structure $J$ on
${\mathfrak p} \cong T_oSU_{2,m}/S(U_2U_m)$ is given by $JX = {\rm ad}(Z)X$ for all $X \in {\mathfrak p}$, where
$Z = \begin{pmatrix} \frac{mi}{m+2}I_2 & 0_{2,m} \\ 0_{m,2} & \frac{-2i}{m+2}I_m \end{pmatrix}$.
In particular, we get
$$JH_t = iH_t \in {\mathfrak p}_{2\epsilon_1} \oplus {\mathfrak p}_{2\epsilon_2}.$$
The maximal complex subbundle ${\mathcal C}_t$ of $M_t$ is invariant under the shape operator of $M_t$
if and only if $JH_t$ is a principal curvature vector. Using the above tables and root space descriptions it is easy to see
that $JH_t$ is a principal curvature vector of $M_t$ if and only if $t \in \{0,\frac{\pi}{4}\}$. These two values for $t$
correspond exactly to the boundary of the closed positive Weyl chamber $\bar{C}^+(\Lambda)$, and
therefore to the two types of singular geodesics on $SU_{2,m}/S(U_2U_m)$.

The quaternionic K\"{a}hler structure ${\mathfrak J}$ on $SU_{2,m}/S(U_2U_m)$ is determined by the
transformations ${\rm ad}(Q)$ on ${\mathfrak p}$ with
$$Q \in \left\{ \left. \begin{pmatrix} A & 0_{2,m} \\
0_{m,2} & 0_{m,m} \end{pmatrix}
\right|  A \in {\mathfrak s}{\mathfrak u}_2 \right\} \subset {\mathfrak k}.$$
We now investigate the maximal quaternionic subbundle ${\mathcal Q}_t$ of $TM_t$. For $t = 0$ we have $H_0 = e_1$ and
$${\mathfrak J}H_0 = {\mathfrak p}_{2\epsilon_1} \oplus ({\mathfrak J}H_0 \cap
({\mathfrak p}_{\epsilon_1 - \epsilon_2}
\oplus {\mathfrak p}_{\epsilon_1 + \epsilon_2})).$$
Using Table \ref{table2} we see that ${\mathfrak J}e_1$ is invariant under the shape operator of $M_0$.
This implies that the maximal quaternionic subbundle ${\mathcal Q}_0$ of $TM_0$ is invariant under the shape
operator of $M_0$. Next, for $t = \frac{\pi}{4}$ we have $H_{\frac{\pi}{4}} = \frac{1}{\sqrt{2}}(e_1 + e_2)$. In this case
we get
$${\mathfrak J}H_{\frac{\pi}{4}} = {\mathfrak p}_{\epsilon_1 + \epsilon_2} \oplus ({\mathfrak J}H_{\frac{\pi}{4}} \cap
({\mathfrak p}_{2\epsilon_1} \oplus {\mathfrak p}_{2\epsilon_2})),$$
which is contained in the $\sqrt{2}$-eigenspace of the shape operator according to Table \ref{table2}.
It follows that the maximal quaternionic subbundle ${\mathcal Q}_{\frac{\pi}{4}}$
of $TM_{\frac{\pi}{4}}$ is invariant under the shape
operator of $M_{\frac{\pi}{4}}$. Finally, for $0 < t < \frac{\pi}{4}$ we see that
$${\mathfrak J}H_t \subset {\mathfrak p}_{2\epsilon_1} \oplus {\mathfrak p}_{2\epsilon_2} \oplus
{\mathfrak p}_{\epsilon_1 - \epsilon_2} \oplus {\mathfrak p}_{\epsilon_1 + \epsilon_2}.$$
We see from Table \ref{table1} and Table \ref{table2} that the four root spaces we just listed
correspond to distinct principal curvatures, and ${\mathfrak J}H_t$ is not equal to the sum of any three of them.
We thus conclude that for $0 < t < \frac{\pi}{4}$
the maximal quaternionic subbundle of $TM_t$ is not invariant under the shape operator of $M_t$.

We finally note that the angle between $JH_t$ and ${\mathfrak J}H_t$
is equal to $2t$. Therefore the horospheres with a singular point at
infinity are characterized by the geometric property that their
normal vectors $H$ satisfy $JH \in {\mathfrak J}H$ or $JH \perp
{\mathfrak J}H$. Since horospheres in $SU_{2,m}/S(U_2U_m)$ are
homogeneous hypersurfaces, and isometries of $SU_{2,m}/S(U_2U_m)$
preserve angles as well as complex and quaternionic subspaces, it
follows that the horospheres with a singular point at infinity can
be characterized by the property that $JH \in {\mathfrak J}H$ or $JH
\perp {\mathfrak J}H$ for some nonzero normal vector. This finishes
the proof of Theorem \ref{horosphereintro}.

\section{Curvature of $SU_{2,m}/S(U_2U_m)$} \label{curvature}

In this section we review some facts about the curvature of the Riemannian symmetric
space $SU_{2,m}/S(U_2U_m)$ equipped with the Killing metric $g$. We denote by $R$ the Riemannian
curvature tensor of $SU_{2,m}/S(U_2U_m)$ with the convention $R(X,Y) = \nabla_X\nabla_Y -
\nabla_Y\nabla_X - \nabla_{[X,Y]}$ for all vector fields $X,Y$ on $SU_{2,m}/S(U_2U_m)$,
where $\nabla$ is the Levi Civita covariant derivative of $SU_{2,m}/S(U_2U_m)$.
Locally the Riemannian curvature tensor $R$ can be expressed entirely in terms of the metric $g$, the
complex structure $J$, and the quaternionic K\"{a}hler structure ${\mathfrak J}$:
\begin{eqnarray*}
R(X,Y)Z & = & -\frac{1}{2} \biggl\lbrack
g(Y,Z)X - g(X,Z)Y  + \ g(JY,Z)JX - g(JX,Z)JY - 2g(JX,Y)JZ \\
& & \qquad + \ \sum_{\nu=1}^3 \left\{g(J_\nu Y,Z)J_\nu X - g(J_\nu X,Z)J_\nu Y -
2g(J_\nu X,Y)J_\nu Z\right\} \\
& & \qquad + \ \sum_{\nu=1}^3 \left\{g(J_\nu JY,Z)J_\nu JX -
g(J_\nu JX,Z)J_\nu JY\right\} \biggr\rbrack\ ,
\end{eqnarray*}
where $J_1,J_2,J_3$ is a canonical local basis of ${\mathfrak J}$. The Riemannian curvature tensor for
the compact symmetric space $SU_{2+m}/S(U_2U_m)$ was calculated explicitly by the first author in
\cite{Be}. The concept of duality between symmetric spaces of compact and noncompact type implies that the
Riemannian curvature tensors of $SU_{2+m}/S(U_2U_m)$ and $SU_{2,m}/S(U_2U_m)$ just differ by sign.
The factor $\frac{1}{2}$ is a consequence of choosing the Killing metric.
The sectional curvature $K$
of the symmetric space $SU_{2,m}/S(U_2U_m)$ equipped with the Killing metric $g$ is bounded by $-4 \leq K \leq 0$.
The sectional curvature $-4$ is obtained for all $2$-planes ${\mathbb C}X$ where $X$ is a nonzero
vector with $JX \in {\mathfrak J}X$.

The Jacobi operator with respect to $X$ is the selfadjoint endomorphism defined by $R_XY = R(Y,X)X$.
We will need later the eigenvalues, eigenspaces and multiplicities of $R_X$ in case $X$ is a singular
unit tangent vector. As we remarked above, there are two types of singular tangent vectors, namely of
type $JX \perp {\mathfrak J}X$ and $JX \in {\mathfrak J}X$. In the second case we can write
$JX = J_1X$ with some almost Hermitian structure $J_1 \in {\mathfrak J}$.

\begin{tb} \label{Jacobi}
The eigenvalues, eigenspaces and multiplicites of the Jacobi operator $R_X$ for a singular
unit tangent vector $X$ are given by
\smallskip
\begin{center}
\begin{tabular}{|l|l|l|l|}
\hline
\mbox{type} &\mbox{eigenvalue} & \mbox{eigenspace} & \mbox{multiplicity} \\
\hline
$JX \perp {\mathfrak J}X$ & $0$ & ${\mathbb R}X \oplus {\mathfrak J}JX$ & $4$\\
& $-\frac{1}{2}$ & $({\mathbb R}X \oplus {\mathbb R}JX \oplus {\mathfrak J}X \oplus {\mathfrak J}JX)^\perp$ & $4m-8$ \\
& $-2$ & ${\mathbb R}JX \oplus {\mathfrak J}X$ & $4$\\
\hline
$JX = J_1X \in {\mathfrak J}X$ & $0$ & ${\mathbb R}X \oplus \{Y \mid Y \perp {\mathbb H}X,JY = -J_1Y\}$ & $2m-1$\\
& $-1$ & $({\mathbb H}X \ominus {\mathbb C}X) \oplus \{Y \mid Y \perp {\mathbb H}X,JY = J_1Y\}$ & $2m$ \\
& $-4$ & ${\mathbb R}JX $ & $1$\\
\hline
\end{tabular}
\end{center}
\end{tb}

\section{The action of $SU_{2,m-1}$ on $SU_{2,m}/S(U_2U_m)$} \label{tube1}

In this section we investigate the action of $SU_{2,m-1}$ on $SU_{2,m}/S(U_2U_m)$. It is clear that
${\mathfrak s}{\mathfrak u}_{2,m-1}$ is invariant under $\theta$, and hence the orbit $W = SU_{2,m-1} \cdot o$ through
$o$ is a totally geodesic submanifold of $SU_{2,m}/S(U_2U_m)$. The isotropy subgroup at $o$
can easily seen to be equal to $S(U_2U_{m-1})$, and therefore $W = SU_{2,m-1}/S(U_2U_{m-1})$.
The tangent space $T_oW$ and the normal space $\nu_oW$ are given by
\begin{eqnarray*}
T_oW & = & \left\{ \left. \begin{pmatrix} 0_{2,2} & C & 0_{2,1} \\
C^* & 0_{m-1,m-1} & 0_{m-1,1} \\
0_{1,2} & 0_{1,m-1} & 0_{1,1} \end{pmatrix}
\right|  C \in M_{2,m-1}({\mathbb C})  \right\}, \\
\nu_oW & = & \left\{ \left. \begin{pmatrix} 0_{2,2} & 0_{2,m-1} & D \\
0_{m-1,2} & 0_{m-1,m-1} & 0_{m-1,1} \\
D^* & 0_{1,m-1} & 0_{1,1} \end{pmatrix}
\right|  D \in M_{2,1}({\mathbb C}) \cong {\mathbb C}^2 \right\},
\end{eqnarray*}
and the isotropy subalgebra is
$$
\left\{ \left. \begin{pmatrix} A & 0_{2,m-1} & 0_{2,1} \\
0_{m-1,2} & B & 0_{m-1,1} \\
0_{1,2} & 0_{1,m-1} & 0_{1,1}   \end{pmatrix}
\right|  A \in {\mathfrak u}_2,\ B \in {\mathfrak u}_{m-1},\ {\rm tr}(A) + {\rm tr}(B) = 0 \right\}.
$$
From this we see that the slice representation of the isotropy subgroup on the normal space $\nu_oW$
is conjugate to the standard $U_2$-action on ${\mathbb C}^2$. Since $U_2$ acts transitively on the
unit sphere in ${\mathbb C}^2$, we conclude that the action of $SU_{2,m-1}$ on $SU_{2,m}/S(U_2U_m)$
is of cohomogeneity one, that is, the codimension of a generic orbit is one. This implies that the principal
orbits of the $SU_{2,m-1}$-action on $SU_{2,m}/S(U_2U_m)$ are the tubes around the totally geodesic
submanifold $W = SU_{2,m-1}/S(U_2U_{m-1})$. We now proceed with calculating the principal curvatures and the
corresponding principal curvature spaces and multiplicities of the tube $W_r$ of radius $r \in {\mathbb R}_+$ around $W$.

Using the explicit description of the complex structure $J$ and the quaternionic K\"{a}hler structure ${\mathfrak J}$
given in Section \ref{horo} we see that the $4$-dimensional normal space $\nu_oW$ is invariant under both $J$ and ${\mathfrak J}$.
This implies that every normal vector $N$ in $\nu_oW$ is singular of type $JN \in {\mathfrak J}N$.
We fix a unit normal vector $N \in \nu_oW$ and denote by $\gamma : {\mathbb R} \to SU_{2,m}/S(U_2U_m)$
the geodesic with $\gamma(0) = o$ and $\dot{\gamma}(0) = N$. The tangent vector $\dot{\gamma}(r)$ is a unit normal
vector of the tube $W_r$ at $\gamma(r)$, and we denote by $A_r$ the shape operator of $W_r$ with
respect to $-\dot{\gamma}(r)$. By $\dot{\gamma}^\perp$ we denote the subbundle of the tangent bundle of
$SU_{2,m}/S(U_2U_m)$ along $\gamma$ consisting of all hyperplanes orthogonal to $\dot{\gamma}$, and by $R_\gamma^\perp$
we denote the Jacobi operator $R(\cdot,\dot{\gamma})\dot{\gamma}$ restricted to $\dot{\gamma}^\perp$.
Now consider the ${\rm End}(\dot{\gamma}^\perp)$-valued ordinary differential equation
$$ Y^{\prime\prime} + R_\gamma^\perp \circ Y = 0\ ,\ Y(0) = \begin{pmatrix} I_{4m-4} & 0_{4m-4,3} \\
0_{3,4m-3} & 0_{3,3}  \end{pmatrix}\ ,\ Y^\prime (0) = \begin{pmatrix} 0_{4m-4,4m-4} & 0_{4m-4,3} \\
0_{3,4m-3} & I_3  \end{pmatrix}, $$
where the decomposition of the matrices is with respect to the parallel translation along $\gamma$ of the decomposition
$\dot{\gamma}^\perp (0) = T_oW \oplus (\nu_oW \cap \dot{\gamma}(0)^\perp)$.
There exists a unique solution $D$ of this differential equation, and the shape operator can be calculated
by means of $A_r = D^\prime(r) \circ D^{-1}(r)$. A straightforward calculation gives the following table

\begin{tb} \label{table4} The principal curvatures and their eigenspaces and multiplicities of the tube $W_r$ with radius $r \in {\mathbb R}_+$
around the totally geodesic submanifold $W = SU_{2,m-1}/S(U_2U_{m-1})$ in $SU_{2,m}/S(U_2U_m)$ are given by
\smallskip
\begin{center}
\begin{tabular}{|l|l|l|}
\hline
\mbox{principal curvature} & \mbox{eigenspace} & \mbox{multiplicity} \\
\hline
$0$ & $\parallel_r\{Y \in T_oW \mid JY = -J_1Y\}$ & $2m-2$\\
$\tanh(r)$ & $\parallel_r\{Y \in T_oW \mid JY = J_1Y\}$ & $2m-2$ \\
$\coth(r)$ & $\parallel_r{\mathbb H}N \ominus {\mathbb C}N$ & $2$\\
$2\coth(2r)$ & $\parallel_r{\mathbb R}JN$ & $1$ \\
\hline
\end{tabular}
\end{center}
\end{tb}

In this table $J_1 \in {\mathfrak J}$ denotes the almost Hermitian structure such that $JN = J_1N$,
and $\parallel_r$ denotes parallel translation along $\gamma$ from $T_oSU_{2,m}/S(U_2U_m)$
to $T_{\gamma(r)}SU_{2,m}/S(U_2U_m)$. We denote by ${\mathcal C}_r$ and ${\mathcal Q}_r$ the
maximal complex subbundle and the maximal quaternionic subbundle of $TW_r$ respectively.
For a principal curvature $\lambda$ we denote by $E_\lambda$ the corresponding eigenspace. Since
both the K\"{a}hler structure and the quaternionic K\"{a}hler structure are invariant under
parallel translation, Table \ref{table4} shows that
\begin{eqnarray*}
{\mathcal Q}_r & = & E_0 \oplus E_{\tanh(r)}, \\
{\mathcal C}_r & = & E_0 \oplus E_{\tanh(r)} \oplus E_{\coth(r)}.
\end{eqnarray*}
This proves that both ${\mathcal C}_r$ and ${\mathcal Q}_r$ are invariant under the shape operator. We
summarize this in

\begin{thm} \label{grassmann}
Let $W_r$ be the tube of radius $r \in {\mathbb R}_+$ around the totally geodesic submanifold
$W = SU_{2,m-1}/S(U_2U_{m-1})$ in $SU_{2,m}/S(U_2U_m)$. The maximal complex subbundle ${\mathcal C}_r$
and the maximal quaternionic subbundle ${\mathcal Q}_r$ of $TW_r$ are both invariant under the
shape operator of $W_r$, that is, $h({\mathcal C}_r,{\mathcal C}_r^\perp) = 0$ and $h({\mathcal Q}_r,{\mathcal Q}_r^\perp) = 0$.
\end{thm}

\section{The action of $Sp_{1,n}$ on $SU_{2,2n}/S(U_2U_{2n})$} \label{tube2}

In this section we investigate the action of $Sp_{1,n}$ on $SU_{2,2n}/S(U_2U_{2n})$. We realize
${\mathfrak s}{\mathfrak p}_{1,n}$ as a subalgebra of ${\mathfrak s}{\mathfrak u}_{2,2n}$ by means
of
$$
{\mathfrak s}{\mathfrak p}_{1,n} = \left\{ \left. \begin{pmatrix}
ix & z & C_1 & C_2 \\
-\bar{z} & -ix & \bar{C}_2 & -\bar{C}_1 \\
C_1^* & \bar{C}_2^* & B_1 & B_2 \\
C_2^* & -\bar{C}_1^* & -\bar{B}_2 & \bar{B}_1 \\
\end{pmatrix} \right|
\begin{matrix}
x \in {\mathbb R},\ z \in {\mathbb C},\ C_1,C_2 \in M_{1,n}({\mathbb C}) \\
B_1 \in {\mathfrak u}_n,\ B_2 \in M_{n,n}({\mathbb C})\ {\rm symmetric}
\end{matrix}
\right\}.
$$
Clearly, ${\mathfrak s}{\mathfrak p}_{1,n}$ in invariant under the Cartan
involution $\theta$ on ${\mathfrak s}{\mathfrak u}_{2,2n}$. Therefore the
orbit $W = Sp_{1,n} \cdot o$ through
$o$ is a totally geodesic submanifold of $SU_{2,2n}/S(U_2U_{2n})$. The isotropy subalgebra at $o$
is
$$
\left\{ \left. \begin{pmatrix}
ix & z & 0_{1,n} & 0_{1.n} \\
-\bar{z} & -ix & 0_{1,n} & 0_{1,n} \\
0_{1,n} & 0_{1,n} & B_1 & B_2 \\
0_{1,n} & 0_{1,n} & -\bar{B}_2 & \bar{B}_1 \\
\end{pmatrix} \right|
\begin{matrix}
x \in {\mathbb R},\ z \in {\mathbb C}, B_1 \in {\mathfrak u}_n, \\
B_2 \in M_{n,n}({\mathbb C})\ {\rm symmetric} \\
\end{matrix}
\right\} ,
$$
and is isomorphic to ${\mathfrak s}{\mathfrak p}_{1} \oplus {\mathfrak s}{\mathfrak p}_n$ by means of
\begin{eqnarray*}
{\mathfrak s}{\mathfrak p}_{1} & = & \left\{ \left. \begin{pmatrix}
ix & z  \\
-\bar{z} & -ix &  \\
\end{pmatrix} \right|
x \in {\mathbb R},\ z \in {\mathbb C} \right\}, \\
{\mathfrak s}{\mathfrak p}_n & = & \left\{ \left. \begin{pmatrix}
B_1 & B_2 \\
-\bar{B}_2 & \bar{B}_1 \\
\end{pmatrix} \right|
B_1 \in {\mathfrak u}_n,\ B_2 \in M_{n,n}({\mathbb C})\ {\rm symmetric}
\right\}.
\end{eqnarray*}
Therefore $W$ is isometric to the $n$-dimensional quaternionic hyperbolic space ${\mathbb H}H^n = Sp_{1,n}/Sp_1Sp_n$.
The tangent space $T_oW$ and the normal space $\nu_oW$ are given by
\begin{eqnarray*}
T_oW & = & \left\{ \left. \begin{pmatrix}
0 & 0 & C_1 & C_2 \\
0 & 0 & \bar{C}_2 & -\bar{C}_1 \\
C_1^* & \bar{C}_2^* & 0_{n,n} & 0_{n,n} \\
C_2^* & -\bar{C}_1^* & 0_{n,n} & 0_{n,n} \\
\end{pmatrix} \right|
C_1,C_2 \in M_{1,n}({\mathbb C})
\right\} \cong {\mathbb H}^n, \\
\nu_oW & = & \left\{ \left. \begin{pmatrix}
0 & 0 & C_1 & C_2 \\
0 & 0 & -\bar{C}_2 & \bar{C}_1 \\
C_1^* & -\bar{C}_2^* & 0_{n,n} & 0_{n,n} \\
C_2^* & \bar{C}_1^* & 0_{n,n} & 0_{n,n} \\
\end{pmatrix} \right|
C_1,C_2 \in M_{1,n}({\mathbb C})
\right\} \cong {\mathbb H}^n.
\end{eqnarray*}

A straightforward calculation shows that the slice representation
of the isotropy subgroup $Sp_1Sp_n$ on $\nu_oW$ is conjugate
to the standard representation of $Sp_1Sp_n$ on ${\mathbb H}^n$.
Since $Sp_1Sp_n$ acts transitively on the
unit sphere in ${\mathbb H}^n$, we conclude that the action of $Sp_{1,n}$ on $SU_{2,2n}/S(U_2U_{2n})$
is of cohomogeneity one. This implies that the principal
orbits of the $Sp_{1,n}$-action on $SU_{2,2n}/S(U_2U_{2n})$ are the tubes around the totally geodesic
submanifold $W = {\mathbb H}H^n$. We now proceed with calculating the principal curvatures and the
corresponding principal curvature spaces and multiplicities of the tube $W_r$ of radius $r \in {\mathbb R}_+$ around $W$.
This can be done as in the previous section. The only difference is that here the ordinary differential
equation is given by
$$ Y^{\prime\prime} + R_\gamma^\perp \circ Y = 0\ ,\ Y(0) = \begin{pmatrix} I_{4n} & 0_{4n,4n-1} \\
0_{4n-1,4n} & 0_{4n-1,4n-1}  \end{pmatrix}\ ,\ Y^\prime (0) = \begin{pmatrix} 0_{4n,4n} & 0_{4n,4n-1} \\
0_{4n-1,4n} & I_{4n-1}  \end{pmatrix}. $$

\begin{tb} \label{table5} The principal curvatures and their eigenspaces and multiplicities of the tube $W_r$ with radius $r \in {\mathbb R}_+$
around the totally geodesic submanifold $W = Sp_{1,n}/Sp_1Sp_n = {\mathbb H}H^n$ in $SU_{2,2n}/S(U_2U_{2n})$ are given by
\smallskip
\begin{center}
\begin{tabular}{|l|l|l|}
\hline
\mbox{principal curvature} & \mbox{eigenspace} & \mbox{multiplicity} \\
\hline
$0$ & $\parallel_r {\mathfrak J}JN$ & $3$\\
$\frac{1}{\sqrt{2}}\tanh(\frac{1}{\sqrt{2}}r)$ & $\parallel_r T_oW \ominus {\mathbb H}JN$ & $4n-4$ \\
$\sqrt{2}\tanh(\sqrt{2}r)$ & $\parallel_r {\mathbb R}JN$ & $1$ \\
$\frac{1}{\sqrt{2}}\coth(\frac{1}{\sqrt{2}}r)$ & $\parallel_r \nu_oW \ominus {\mathbb H}N$ & $4n-4$ \\
$\sqrt{2}\coth(\sqrt{2}r)$ & $\parallel_r {\mathfrak J}N$ & $3$ \\
\hline
\end{tabular}
\end{center}
\end{tb}

We denote by ${\mathcal C}_r$ and ${\mathcal Q}_r$ the
maximal complex subbundle and the maximal quaternionic subbundle of $TW_r$ respectively.
By inspection of Table \ref{table5} we obtain
\begin{eqnarray*}
{\mathcal Q}_r & = & E_0 \oplus E_{\frac{1}{\sqrt{2}}\tanh(\frac{1}{\sqrt{2}}r)}
\oplus E_{\sqrt{2}\tanh(\sqrt{2}r)} \oplus E_{\frac{1}{\sqrt{2}}\coth(\frac{1}{\sqrt{2}}r)}, \\
{\mathcal C}_r & = & E_0 \oplus E_{\frac{1}{\sqrt{2}}\tanh(\frac{1}{\sqrt{2}}r)}
\oplus E_{\frac{1}{\sqrt{2}}\coth(\frac{1}{\sqrt{2}}r)} \oplus E_{\sqrt{2}\coth(\sqrt{2}r)}.
\end{eqnarray*}
This proves that both ${\mathcal C}_r$ and ${\mathcal Q}_r$ are invariant under the shape operator. We
summarize this in

\begin{thm} \label{hyperbolic}
Let $W_r$ be the tube of radius $r \in {\mathbb R}_+$ around the totally geodesic submanifold
$W = Sp_{1,n}/Sp_1Sp_n = {\mathbb H}H^n$ in $SU_{2,2n}/S(U_2U_{2n})$. The maximal complex subbundle ${\mathcal C}_r$
and the maximal quaternionic subbundle ${\mathcal Q}_r$ of $TW_r$ are both invariant under the
shape operator of $W_r$, that is, $h({\mathcal C}_r,{\mathcal C}_r^\perp) = 0$ and $h({\mathcal Q}_r,{\mathcal Q}_r^\perp) = 0$.
\end{thm}

\section{Proof of Theorem \ref{mainresult}} \label{proof}

Let $M$ be a connected hypersurface in $SU_{2,m}/S(U_2U_m)$ and assume that the maximal
complex subbundle ${\mathcal C}$ and the maximal quaternionic subbundle ${\mathcal Q}$ of $TM$ are
invariant under the shape operator $A$ of $M$, that is,
$h({\mathcal C},{\mathcal C}^\perp) = 0$ and $h({\mathcal Q},{\mathcal Q}^\perp) = 0$.
The induced Riemannian metric on $M$ will also be denoted by $g$,
and $\nabla$ denotes the Riemannian connection of $(M,g)$.
Let $N$ be a local unit normal field of $M$ and $A$ the
shape operator of $M$ with respect to $N$.
Since all our calculations are of local nature, we will assume for simplicity that $N$ and other
objects as local canonical bases of ${\mathfrak J}$ are globally defined on $M$.
The K\"{a}hler structure $J$ of $SU_{2,m}/S(U_2U_m)$ induces on
$M$ an almost contact metric structure
$(\phi,\xi,\eta,g)$, where the vector field $\xi$ on $M$ is defined by $\xi = -JN$,
the one-form $\eta$ on $M$ is defined by $\eta(X) = g(X,\xi)$, and the tensor field
$\phi$ on $M$ is defined by $\phi X = JX - \eta(X)N$. Furthermore, let $J_1,J_2,J_3$ be a canonical
local basis of ${\mathfrak J}$. Then each $J_\nu$ induces an almost contact
metric structure $(\phi_\nu,\xi_\nu,\eta_\nu,g)$ on $M$.
The following identities are easy to establish and are used frequently throughout this section:
$$
\phi_{\nu+1}\xi_\nu = -\xi_{\nu+2}\ ,\ \phi_\nu\xi_{\nu+1} =
 \xi_{\nu+2}\ ,\
\phi\xi_\nu = \phi_\nu\xi\ ,\ \eta_\nu(\phi X) = \eta(\phi_\nu
 X)\ .
$$
Here, and below, the index $\nu$ is to be taken modulo $3$. Using the explicit expression
for the Riemannian curvature tensor of $SU_{2,m}/S(U_2U_m)$ given in Section \ref{curvature}, we can
write the Codazzi equation as
\begin{eqnarray*}
(\nabla_XA)Y - (\nabla_YA)X
&=& -\frac{1}{2} \biggl\lbrack \eta(X)\phi Y - \eta(Y)\phi X - 2g(\phi X,Y)\xi \\
&& \qquad + \sum_{\nu=1}^3\{ \eta_\nu(X)\phi_\nu Y - \eta_\nu(Y)\phi_\nu
 X - 2g(\phi_\nu X,Y)\xi_\nu \} \\
&& \qquad + \sum_{\nu=1}^3 \{\eta(\phi_\nu X)\phi_\nu\phi Y
- \eta(\phi_\nu Y)\phi_\nu\phi X \} \\
&& \qquad + \sum_{\nu=1}^3 \{ \eta(X)\eta(\phi_\nu Y)
- \eta(Y)\eta(\phi_\nu X)\} \xi_\nu\ \biggl\rbrack.
\end{eqnarray*}
The Codazzi equation is a substantial tool for establishing relations between principal curvatures.

\begin{prop} \label{pcC}
Assume that the maximal complex subbundle ${\mathcal C}$ of $TM$ is
invariant under the shape operator $A$ of $M$. Then $\xi$ is a principal curvature
vector field on $M$, say $A\xi = \alpha\xi$. Moreover, if $X \in {\mathcal C}$ is a principal curvature vector of $M$,
say $AX = \lambda X$, then
$$
(2\lambda-\alpha)A\phi X + (1 - \alpha\lambda)\phi X =
\sum_{\nu=1}^3 \{ 2\eta(\xi_\nu)\eta(\phi_\nu X)\xi -
\eta_\nu(X)\phi_\nu\xi  - \eta(\phi_\nu X)\xi_\nu -
\eta(\xi_\nu)\phi_\nu X \} .
$$
\end{prop}

\begin{proof}
Since the tangent bundle $TM$ decomposes orthogonally into $TM = {\mathbb R}\xi \oplus {\mathcal C}$, it is clear that
the assumption $A{\mathcal C} \subset {\mathcal C}$ implies $A\xi = \alpha\xi$ for some smooth function $\alpha$ on $M$.
Using the Codazzi equation we get for arbitrary tangent vector fields $X$ and $Y$ that
\begin{eqnarray*}
&& g(\phi X,Y) - \sum_{\nu = 1}^3 \{\eta_\nu(X)\eta(\phi_\nu Y)
- \eta_\nu(Y)\eta(\phi_\nu X) - g(\phi_\nu X,Y)\eta(\xi_\nu)\} \\
&=& g((\nabla_XA)Y - (\nabla_YA)X, \xi) \\
&=& g((\nabla_XA)\xi,Y) - g((\nabla_YA)\xi,X) \\
&=& (X\alpha)\eta(Y) - (Y\alpha)\eta(X) + \alpha g((A\phi + \phi
A)X,Y) - 2g(A\phi AX,Y) .
\end{eqnarray*}
For $X = \xi$ this equation yields
\begin{equation}
Y\alpha = (\xi\alpha)\eta(Y) + 2\sum_{\nu = 1}^3
\eta(\xi_\nu)\eta(\phi_\nu Y). \label{gradient}
\end{equation}
Inserting this and the corresponding equation for $X\alpha$ into the previous equation gives
\begin{eqnarray*}
&& g(\phi X,Y) - \sum_{\nu = 1}^3 \{ \eta_\nu(X)\eta(\phi_\nu Y) -
\eta_\nu(Y)\eta(\phi_\nu X)
- g(\phi_\nu X,Y)\eta(\xi_\nu)\} \\
& =&   2\sum_{\nu=1}^3 \{ \eta(Y)\eta(\phi_\nu X) -
\eta(X)\eta(\phi_\nu Y) \} \eta(\xi_\nu) +  \alpha g((A\phi + \phi
A)X,Y) - 2g(A\phi AX,Y) .
\end{eqnarray*}
If we now insert $X \in {\mathcal C}$ with $AX = \lambda X$, the equation in Proposition \ref{pcC}
follows easily.
\end{proof}

Since the quaternionic K\"{a}hler structure structure ${\mathfrak J}$ is invariant under parallel
translation on $SU_{2,m}/S(U_2U_m)$, there exist one-forms $q_1,q_2,q_3$ such that
$$
\bar\nabla_X J_\nu = q_{\nu+2}(X)J_{\nu+1} - q_{\nu+1}(X)J_{\nu+2}
$$
for all vector fields $X$ on $SU_{2,m}/S(U_2U_m)$, where $\bar\nabla$ is the Levi Civita covariant derivative
on $SU_{2,m}/S(U_2U_m)$.

\begin{prop} \label{pcQ}
Assume that the maximal quaternionic subbundle ${\mathcal Q}$ of $TM$ is
invariant under the shape operator $A$ of $M$. Then there exists a canonical local basis $J_1,J_2,J_3$
of ${\mathfrak J}$ such that $\xi_\nu$ is a principal curvature vector field of $M$, say $A\xi_\nu = \beta_\nu \xi_\nu$.
Moreover, if $X \in {\mathcal Q}$ is a principal curvature vector of $M$,
say $AX = \lambda X$, then
\begin{eqnarray*}
&& (2\lambda - \beta_\nu)A\phi_\nu X + (1-\lambda\beta_\nu)\phi_\nu X \\
&= & \{2\eta(\xi_\nu)\eta(\phi_\nu X) - \eta(\xi_{\nu+1})\eta(\phi_{\nu+1} X)
- \eta(\xi_{\nu+2})\eta(\phi_{\nu+2} X)\}\xi_\nu \\
&& + (\beta_\nu - \beta_{\nu+1})q_{\nu+2}(X)\xi_{\nu+1}
- (\beta_\nu - \beta_{\nu+2})q_{\nu+1}(X)\xi_{\nu+2} \\
&& - \eta(\xi_\nu)\phi X - \eta(\phi_\nu X)\xi
- \eta(X)\phi\xi_\nu + \eta(\phi_{\nu+2} X)\phi\xi_{\nu+1}
- \eta(\phi_{\nu+1} X)\phi\xi_{\nu+2}
\end{eqnarray*}
holds for all $\nu \in \{1,2,3\}$.
\end{prop}

\begin{proof}
Since the tangent bundle $TM$ decomposes orthogonally into $TM = {\mathfrak J}N \oplus {\mathcal Q}$, it is clear that
the assumption $A{\mathcal Q} \subset {\mathcal Q}$ implies that there exists a canonical local basis $J_1,J_2,J_3$
of ${\mathfrak J}$ such that $A\xi_\nu = \beta_\nu \xi_\nu$ for some functions $\beta_1,\beta_2,\beta_3$ on $M$.
Using the Codazzi equation we get for arbitrary tangent vector fields $X$ and $Y$ that
\begin{eqnarray*}
&& g(\phi X,Y)\eta(\xi_\nu) + g(\phi_\nu X,Y) -\eta(X)\eta_\nu(\phi Y) + \eta(Y)\eta_\nu(\phi X) \\
&&  - \eta_{\nu+1}(X)\eta_{\nu+2}(Y) + \eta_{\nu+1}(Y)\eta_{\nu+2}(X) \\
&& - \eta(\phi_{\nu+1} X)\eta(\phi_{\nu+2} Y) + \eta(\phi_{\nu+1} Y)\eta(\phi_{\nu+2} X)  \\
& = & g((\nabla_XA)Y - (\nabla_YA)X, \xi_\nu) \\
& = & g((\nabla_XA)\xi_\nu,Y) - g((\nabla_YA)\xi_\nu,X) \\
& = & (X\beta_\nu)\eta_\nu(Y) - (Y\beta_\nu)\eta_\nu(X) + \beta_\nu g((A\phi_\nu + \phi_\nu A)X,Y)
- 2g(A\phi_\nu AX,Y) \\
&& + (\beta_\nu - \beta_{\nu+1})\{q_{\nu+2}(X)\eta_{\nu+1}(Y)
 - q_{\nu+2}(Y)\eta_{\nu+1}(X)\} \\
&& - (\beta_\nu - \beta_{\nu+2})\{q_{\nu+1}(X)\eta_{\nu+2}(Y)
 - q_{\nu+1}(Y)\eta_{\nu+2}(X)\} .
\end{eqnarray*}
For $X = \xi_\nu$ this equation yields
\begin{eqnarray*}
Y\beta_\nu & = & (\xi_\nu\beta_\nu)\eta_\nu(Y) +
2\eta(\xi_\nu)\eta(\phi_\nu Y)
- \eta(\xi_{\nu+1})\eta(\phi_{\nu+1} Y) - \eta(\xi_{\nu+2})\eta(\phi_{\nu+2} Y) \\
&& + (\beta_\nu - \beta_{\nu+1})q_{\nu+2}(\xi_\nu)\eta_{\nu+1}(Y)
- (\beta_\nu - \beta_{\nu+2})q_{\nu+1}(\xi_\nu)\eta_{\nu+2}(Y) .
\end{eqnarray*}
Inserting this and the corresponding equation for $X\beta_\nu$ into the previous equation gives
\begin{eqnarray*}
&& g(\phi X,Y)\eta(\xi_\nu) + g(\phi_\nu X,Y) -\eta(X)\eta(\phi_\nu Y) + \eta(Y)\eta(\phi_\nu X) \\
&&  - \eta_{\nu+1}(X)\eta_{\nu+2}(Y) + \eta_{\nu+1}(Y)\eta_{\nu+2}(X) \\
&& - \eta(\phi_{\nu+1} X)\eta(\phi_{\nu+2} Y) + \eta(\phi_{\nu+1} Y)\eta(\phi_{\nu+2} X)  \\
& =& \beta_\nu g((A\phi_\nu + \phi_\nu A)X,Y) - 2g(A\phi_\nu AX,Y) \\
&& -\eta(\xi_{\nu+1})\{\eta(\phi_{\nu+1} X)\eta_\nu(Y) - \eta(\phi_{\nu+1} Y)\eta_\nu(X)\} \\
&& - \eta(\xi_{\nu+2})\{ \eta(\phi_{\nu+2} X)\eta_\nu(Y) - \eta(\phi_{\nu+2} Y)\eta_\nu(X)\} \\
&& - 2\eta(\xi_\nu)\{\eta(\phi_\nu Y)\eta_\nu(X) - \eta(\phi_\nu X)\eta_\nu(Y)\big\} \\
&& + (\beta_\nu - \beta_{\nu+1}) \{q_{\nu+2}(\xi_\nu)(\eta_{\nu+1}(X)\eta_\nu(Y) - \eta_{\nu+1}(Y)\eta_\nu(X)) \\
&& \qquad \qquad \qquad + q_{\nu+2}(X)\eta_{\nu+1}(Y) - q_{\nu+2}(Y)\eta_{\nu+1}(X)\} \\
&& - (\beta_\nu - \beta_{\nu+2}) \{q_{\nu+1}(\xi_\nu)(\eta_{\nu+2}(X)\eta_\nu(Y) - \eta_{\nu+2}(Y)\eta_\nu(X)) \\
&& \qquad \qquad \qquad + q_{\nu+1}(X)\eta_{\nu+2}(Y) - q_{\nu+1}(Y)\eta_{\nu+2}(X)\big\}.
\end{eqnarray*}
If we now insert $X \in {\mathcal Q}$ with $AX = \lambda X$, the equation in Proposition \ref{pcQ}
follows by a straightforward calculation.
\end{proof}

We will now combine the two previous results.

\begin{prop} \label{key}
Assume that the maximal complex subbundle ${\mathcal C}$ of $TM$ and the
maximal quaternionic subbundle ${\mathcal Q}$ of $TM$ are both
invariant under the shape operator $A$ of $M$. Then the normal bundle $\nu M$
of $M$ consists of singular tangent vectors of $SU_{2,m}/S(U_2U_m)$.
\end{prop}

\begin{proof}
By Propositions \ref{pcC} and \ref{pcQ} we can assume that $A\xi =
\alpha \xi$ and $A\xi_\nu = \beta_\nu\xi_\nu$. Since $TM$ decomposes
orthogonally into $TM = {\mathfrak J}N \oplus {\mathcal Q}$, there
exist unit vectors $Z \in {\mathfrak J}N$ and $X \in {\mathcal Q}$
such that $\xi = \eta(Z)Z + \eta(X)X$. We then get
$$
\alpha\eta(Z)Z + \alpha\eta(X)X = \alpha\xi = A\xi = \eta(Z)AZ +
\eta(X)AX.
$$
Since both ${\mathfrak J}N$ and ${\mathcal Q}$ are invariant under
$A$, we have $AZ \in {\mathfrak J}N$ and $AX \in {\mathcal Q}$.

Assume that $\eta(X)\eta(Z) \neq 0$. The previous equation then
implies $AZ = \alpha Z$ and $AX = \alpha X$, and thus we can insert
$X$ into the equation of Proposition \ref{pcQ}. Without loss of
generality we may assume that $Z = \xi_3$. Taking the inner product
of the equation for the index $\nu = 1$ in Proposition \ref{pcQ}
with $\xi_2$ leads to
$$
 \eta(X)\eta(Z) = (\beta_2 - \beta_1)q_3(X) ,
$$
and taking the inner product of the equation for the index $\nu = 2$
in Proposition \ref{pcQ} with  $\xi_1$ gives
$$
  \eta(X)\eta(Z) = (\beta_1 - \beta_2)q_3(X) ,
$$
Adding up the previous two equations yields $\eta(X)\eta(Z) = 0$,
which contradicts the assumption $\eta(X)\eta(Z) \neq 0$. Therefore
we must have $\eta(X)\eta(Z) = 0$, which means that $\xi$ is tangent
to ${\mathfrak J}N$ or tangent to ${\mathcal Q}$. Since $\xi = -JN$
this implies that either $JN \in {\mathfrak J}N$ or $JN \in
{\mathcal Q} \perp {\mathfrak J}N$, and we conclude that $N$ is a
singular tangent vector of $SU_{2,m}/S(U_2U_m)$ at each point.
\end{proof}

We proceed by investigating separately the two types of singular
tangent vectors.

\subsection{The case $JN \perp {\mathfrak J}N$} \label{case2}

In this subsection we assume that $JN \perp {\mathfrak J}N$. In this
situation the vector fields
$\xi,\xi_1,\xi_2,\xi_3,\phi\xi_1,\phi\xi_2,\phi\xi_3$ are
orthonormal.

\begin{lm} \label{lemma1}
For each $\nu \in \{1,2,3\}$ we have $A\phi\xi_\nu = \gamma_\nu\phi\xi_\nu$, and one of the following two cases holds:
\begin{itemize}
\item[(i)] $\alpha\beta_\nu = 2$ and $\gamma_\nu = 0$,
\item[(ii)] $\alpha = \beta_\nu \neq 0$ and $\gamma_\nu = \frac{\alpha^2 -
2}{\alpha}$.
\end{itemize}
\end{lm}

\begin{proof} When we insert $X = \xi_\nu$ with $A\xi_\nu =
\beta_\nu\xi_\nu$ into the equation in Proposition \ref{pcC}, we get
$$
(2\beta_\nu - \alpha) A\phi\xi_\nu = (\alpha \beta_\nu -
2)\phi\xi_\nu,
$$
and when we insert $X = \xi$ with $A\xi = \alpha\xi$ into the
equation in Proposition \ref{pcQ}, we get
$$
(2\alpha - \beta_\nu) A\phi\xi_\nu = (\alpha \beta_\nu -
2)\phi\xi_\nu.
$$
These two equations imply (i) for $\alpha\beta_\nu = 2$ and (ii) for $\alpha\beta_\nu \neq 2$.
\end{proof}

\begin{lm} \label{lemma2}
We have $\beta_1 = \beta_2 = \beta_3 =:\beta$, $\gamma_1 = \gamma_2 = \gamma_3 =: \gamma = 0$ and $\alpha\beta = 2$.
\end{lm}

\begin{proof}
By evaluating the equation in Proposition \ref{pcQ} for $X = \phi\xi_{\nu+1}$
and $\lambda = \gamma_{\nu+1}$ we get
\begin{equation} \label{eq1}
2\gamma_{\nu+1}\gamma_{\nu+2} = \beta_\nu(\gamma_{\nu+1} + \gamma_{\nu+2})
\end{equation}
for all $\nu \in \{1,2,3\}$. Since $\beta_\nu \neq 0$ for all $\nu \in \{1,2,3\}$ according
to Lemma \ref{lemma1}, this implies that either $\gamma_1 = \gamma_2 = \gamma_3 = 0$ or
$\gamma_\nu \neq 0$ for all $\nu \in \{1,2,3\}$.
Assume that $\gamma_\nu \neq 0$ for all $\nu \in \{1,2,3\}$. Then, using Lemma \ref{lemma1},
we have $\gamma_\nu = \frac{\alpha^2 -
2}{\alpha}$ and $\beta_\nu = \alpha$ for all $\nu \in \{1,2,3\}$, and inserting this into (\ref{eq1})
yields $\alpha^2 = 2$. However, $\alpha^2 = 2$ implies $\gamma_\nu = 0$, which contradicts our assumption
that $\gamma_\nu \neq 0$ for all $\nu \in \{1,2,3\}$. We therefore must have
$\gamma_1 = \gamma_2 = \gamma_3 = 0$, and the assertion then follows from Lemma \ref{lemma1}.
\end{proof}

We now derive some equations for the principal curvatures
corresponding to principal curvature vectors which are orthogonal to
${\mathbb R}JN \oplus {\mathfrak J}N \oplus {\mathfrak J}JN$.
Note that the orthogonal complement of ${\mathbb R}JN \oplus {\mathfrak J}N \oplus
{\mathfrak J}JN$ in $TM$ is equal to ${\mathcal C} \cap {\mathcal Q} \cap J{\mathcal Q}$.

\begin{lm} \label{lemma3}
Let $X \in {\mathcal C} \cap {\mathcal Q} \cap J{\mathcal Q}$ with $AX = \lambda X$. Then we have
\begin{eqnarray}
(2 \lambda - \alpha)A\phi X & = & (\alpha \lambda - 1) \phi X,
\label{eq3} \\
2(\alpha\lambda - 1)A\phi_\nu X & = & (2\lambda - \alpha)\phi_\nu X.
\label{eq4}
\end{eqnarray}
\end{lm}

\begin{proof} Equation (\ref{eq3}) follows from Proposition
\ref{pcC}, and (\ref{eq4}) follows from Proposition
\ref{pcQ} using the fact that $\alpha\beta = 2$ according to Lemma
\ref{lemma2}.
\end{proof}

If we assume $2\lambda - \alpha = 0$, we get $\alpha\lambda = 1$
from (\ref{eq3}) and therefore $\alpha^2 = 2$. Since $\alpha\beta =
2$ this implies $\alpha = \beta$. For this reason we consider the
two cases $\alpha \neq \beta$ and $\alpha = \beta$ separately.

We first assume that $\alpha \neq \beta$. Then we must have
$2\lambda - \alpha \neq 0$ and therefore also $\alpha\lambda - 1
\neq 0$ by (\ref{eq4}). From (\ref{eq4}) we then get
$$
A\phi_{\nu+1} X = \frac{2\lambda - \alpha}{2(\alpha\lambda -
1)}\phi_{\nu+1} X.
$$
Applying (\ref{eq4}) to $\phi_{\nu+1}X$ we obtain
$$
A\phi_\nu\phi_{\nu+1}X = \lambda\phi_\nu\phi_{\nu+1}X.
$$
On the other hand, by (\ref{eq4}) we also have
$$
A\phi_{\nu+2} X = \frac{2\lambda - \alpha}{2(\alpha\lambda -
1)}\phi_{\nu+2} X.
$$
Since $\phi_\nu\phi_{\nu+1}X = \phi_{\nu+2}X$, the previous two
equations imply that $\lambda$ is a solution of the quadratic
equation
$$
2\alpha\lambda^2 - 4\lambda + \alpha = 0.
$$
This shows that $A$ restricted to ${\mathcal C} \cap {\mathcal Q} \cap J{\mathcal Q}$
has at most two eigenvalues.
Moreover, each solution of $2\alpha\lambda^2 - 4\lambda + \alpha =
0$ satisfies $\frac{2\lambda - \alpha}{2(\alpha\lambda - 1)} =
\lambda$, which means that the corresponding eigenspace is
${\mathfrak J}$-invariant. From (\ref{eq3}) we see that
$\phi X$ is a principal curvature vector with principal curvature
$\frac{\alpha\lambda - 1}{2\lambda - \alpha}$. If we assume that
$\frac{\alpha\lambda - 1}{2\lambda - \alpha} = \lambda$, we get
$2\lambda^2 - 2\alpha\lambda + \alpha = 0$, which together with
$2\alpha\lambda^2 - 4\lambda + \alpha = 0$ leads to $\lambda = 0$
(which leads to $\alpha = 0$ and contradicts $\alpha\beta = 2$) or
$\alpha^2 = 2$ (which because of $\alpha\beta = 2$ leads to $\alpha
= \beta$ and contradicts the assumption $\alpha \neq \beta$).
Therefore we must have $\frac{\alpha\lambda - 1}{2\lambda - \alpha}
\neq \lambda$. Altogether this shows that $A$ restricted to ${\mathcal C} \cap {\mathcal Q} \cap J{\mathcal Q}$
has precisely two eigenvalues, namely the two solutions
$\lambda_1$ and $\lambda_2$ of $2\alpha\lambda^2 - 4\lambda + \alpha
= 0$, and that $J$ maps the two eigenspaces onto each other. It is
easy to see that $\lambda_1 + \lambda_2 = \frac{2}{\alpha} = \beta
\neq 0$.

From (\ref{gradient}) we see that the gradient ${\rm
grad}^\alpha$ of $\alpha$ on $M$ satisfies ${\rm grad}^\alpha =
(\xi\alpha)\xi$. For the Hessian ${\rm hess}^\alpha(X,Y)$ of
$\alpha$ we then get
\begin{eqnarray*}
{\rm hess}^\alpha(X,Y) & = & g(\nabla_X((\xi\alpha)\xi) , Y) \\
& = & X(\xi\alpha)\eta(Y) + (\xi\alpha)g(\nabla_X\xi ,Y) \\ & = &
X(\xi\alpha)\eta(Y) + (\xi\alpha)g(\phi AX ,Y).
\end{eqnarray*}
Since the Hessian of a function is symmetric, this implies
$$
X(\xi\alpha)\eta(Y) - Y(\xi\alpha)\eta(X) + (\xi\alpha)g((A \phi +
\phi A)X ,Y) = 0.
$$
Inserting $X = \xi$ yields $Y(\xi\alpha) = \xi(\xi\alpha)\eta(Y)$,
and inserting this and the corresponding equation for $X(\xi\alpha)$
into the previous one shows that
$$
(\xi\alpha)g((A \phi + \phi A)X ,Y) = 0.
$$
As we have seen above, on ${\mathcal C} \cap {\mathcal Q} \cap J{\mathcal Q}$
we get $(A\phi + \phi A)X =
\beta X$, and since $\beta \neq 0$ we conclude that $\xi\alpha = 0$
and hence ${\rm grad}^\alpha = (\xi\alpha)\xi= 0$. Since $M$ is
connected we obtain that $\alpha$ is constant. From
$2\alpha\lambda^2 - 4\lambda + \alpha = 0$ we get $\alpha^2 < 2$,
and hence we can write $\alpha = \sqrt{2}\tanh(\sqrt{2}r)$ for some
positive real number $r$ and some suitable orientation of the normal
vector. Writing $\alpha$ in this way, the two solutions of
$2\alpha\lambda^2 - 4\lambda + \alpha = 0$ can be written as
$\lambda_1 = \frac{1}{\sqrt{2}}\tanh(\frac{1}{\sqrt{2}}r)$ and
$\lambda_2 = \frac{1}{\sqrt{2}}\coth(\frac{1}{\sqrt{2}}r)$. From
$\alpha\beta = 2$ we also get $\beta = \sqrt{2}\coth(\sqrt{2}r)$.

We now assume that $\alpha = \beta$. Since $\alpha\beta = 2$ we may
assume that $\alpha = \sqrt{2}$ (by a suitable orientation of the
normal vector). Assume that there exists a principal curvature
$\lambda$ of $A$ restricted to ${\mathcal C} \cap {\mathcal Q} \cap J{\mathcal Q}$
such that $\lambda \neq
\frac{1}{\sqrt{2}}$. From (\ref{eq3}) we then get
$$
A\phi X = \frac{\alpha\lambda - 1}{2\lambda - \alpha}\phi X =
\frac{1}{\sqrt{2}}\phi X,
$$
and from (\ref{eq4}) we obtain
$$
A\phi_\nu X = \frac{2\lambda - \alpha}{2(\alpha\lambda - 1)}\phi_\nu
X = \frac{1}{\sqrt{2}}\phi_\nu X.
$$

Thus we have proved:

\begin{prop} \label{pccase2}
One of the following three cases holds:
\begin{itemize}
\item[(i)] $M$ has five (four for $r = \sqrt{2}{\rm Artanh}(1/\sqrt{3})$ in which case $\alpha = \lambda_2$)
distinct constant principal curvatures
\begin{eqnarray*}
\alpha = \sqrt{2}\tanh(\sqrt{2}r)\ ,\ \beta =
\sqrt{2}\coth(\sqrt{2}r)\ ,\ \gamma = 0\ ,& &\\  \lambda_1 =
\frac{1}{\sqrt{2}}\tanh(\frac{1}{\sqrt{2}}r)\ ,\ \lambda_2 =
\frac{1}{\sqrt{2}}\coth(\frac{1}{\sqrt{2}}r), & &
\end{eqnarray*}
and the corresponding principal curvature spaces are
$$
T_\alpha = {\mathcal C}^\perp\ ,\ T_\beta = {\mathcal Q}^\perp\ ,\ T_\gamma =
J{\mathcal Q}^\perp = JT_\beta.
$$
The principal curvature spaces $T_{\lambda_1}$ and $T_{\lambda_2}$
are invariant under ${\mathfrak J}$ and are mapped onto each other
by $J$. In particular, the quaternionic dimension of
$SU_{2,m}/S(U_2U_m)$ must be even.
\item[(ii)] $M$ has exactly three distinct constant principal curvatures
$$
\alpha = \beta = \sqrt{2}\ ,\ \gamma = 0\ ,\ \lambda =
\frac{1}{\sqrt{2}}
$$
with corresponding principal curvature spaces
$$
T_\alpha = ({\mathcal C} \cap {\mathcal Q})^\perp\ ,\ T_\gamma =
J{\mathcal Q}^\perp\ ,\ T_\lambda = {\mathcal C} \cap {\mathcal Q} \cap J{\mathcal Q}.
$$
\item[(iii)] $M$ has at least four distinct principal curvatures,
three of which are given by
$$
\alpha = \beta = \sqrt{2}\ ,\ \gamma = 0\ ,\ \lambda =
\frac{1}{\sqrt{2}}
$$
with corresponding principal curvature spaces
$$
T_\alpha = ({\mathcal C} \cap {\mathcal Q})^\perp\ ,\ T_\gamma =
J{\mathcal Q}^\perp\ ,\ T_\lambda \subset {\mathcal C} \cap {\mathcal Q} \cap J{\mathcal Q}.
$$
If $\mu$ is another (possibly nonconstant) principal
curvature function  , then $JT_{\mu} \subset T_\lambda$
and ${\mathfrak J}T_{\mu} \subset T_\lambda$.
\end{itemize}
\end{prop}

Assume that $M$ satisfies property (i) in Proposition \ref{pccase2}.
Then the quaternionic dimension of $SU_{2,m}/S(U_2U_m)$ is even, say
$m = 2n$. For $p \in M$ we denote by $c_p : {\mathbb R} \to
SU_{2,2n}/S(U_2U_{2n})$ the geodesic in $SU_{2,2n}/S(U_2U_{2n})$
with $c_p(0) = p$ and $\dot{c}_p(0) = N_p$, and define the smooth
map
$$
F : M \to SU_{2,2n}/S(U_2U_{2n}), p \mapsto c_p(r).
$$
Geometrically, $F$ is the displacement of $M$ at distance $r$ in
direction of the unit normal vector field $N$. For each $p \in M$
the differential $d_pF$ of $F$ at $p$ can be computed using Jacobi
vector fields by means of $d_pF(X) = Z_X(r)$, where $Z_X$ is the
Jacobi vector field along $c_p$ with initial value $Z_X(0) = X$ and
$Z_X^\prime(0) = -AX$. Using the explicit description of the Jacobi
operator $R_N$ given in Table \ref{Jacobi} for the case $JN \perp
{\mathfrak J}N$ we get
$$
Z_X(r) = \begin{cases} E_X(r) & ,\ {\rm if}\ X \in T_\gamma,\\
\left(\cosh\left(\frac{1}{\sqrt{2}}r\right) -
\sqrt{2}\kappa\sinh\left(\frac{1}{\sqrt{2}}r\right)\right)E_X(r) &
,\ {\rm if}\ X \in T_\kappa\ {\rm and}\ \kappa \in
\{\lambda_1,\lambda_2\},\\
\left(\cosh\left(\sqrt{2}r\right) -
\frac{\kappa}{\sqrt{2}}\sinh\left(\sqrt{2}r\right)\right)E_X(r) & ,\
{\rm if}\ X \in T_\kappa\ {\rm and}\ \kappa \in \{\alpha,\beta\},
\end{cases}
$$
where $E_X$ denotes the parallel vector field along $c_p$ with
$E_X(0) = X$. This shows that the kernel ${\rm ker}\,dF$ of $dF$ is
given by
$$
{\rm ker}\,dF = T_\beta \oplus T_{\lambda_2} = {\mathcal Q}^\perp \oplus
T_{\lambda_2},
$$
and that $F$ is of constant rank equal to the rank of the vector
bundle $T_\alpha \oplus T_\gamma \oplus T_{\lambda_1}$, which is
equal to $4n$. Thus, locally, $F$ is a submersion onto a
$4n$-dimensional submanifold $B$ of $SU_{2,2n}/S(U_2U_{2n})$.
Moreover, the tangent space of $B$ at $F(p)$ is obtained by parallel
translation of $(T_\alpha \oplus T_\gamma \oplus T_{\lambda_1})(p) =
({\mathbb H}\xi \oplus T_{\lambda_1})(p)$, which is a quaternionic
(with respect to ${\mathfrak J}$) and real (with respect to $J$)
subspace of $T_pSU_{2,2n}/S(U_2U_{2n})$. Since both $J$ and
${\mathfrak J}$ are parallel along $c_p$, also $T_{F(p)}B$ is a
quaternionic (with respect to ${\mathfrak J}$) and real (with
respect to $J$) subspace of $T_{F(p)}SU_{2,2n}/S(U_2U_{2n})$. Thus
$B$ is a quaternionic and real submanifold of
$SU_{2,2n}/S(U_2U_{2n})$. Since every quaternionic submanifold of a
quaternionic K\"{a}hler manifold is necessarily totally geodesic
(see e.g.\ \cite{G}), we see that $B$ is a totally geodesic
submanifold of $SU_{2,2n}/S(U_2U_{2n})$. The well-known concept of
duality between symmetric spaces of noncompact type and symmetric
spaces of compact type establishes a one-to-one correspondence
between totally geodesic submanifolds of a symmetric space of
noncompact type and its dual symmetric space of compact type. Using
the concept of duality between the symmetric spaces
$SU_{2,2n}/S(U_2U_{2n})$ and $SU_{2+2n}/S(U_2U_{2n})$, it follows
from the classification of totally geodesic submanifolds in complex
$2$-plane Grassmannians (see \cite{K}), that $B$ is an open part of a
quaternionic hyperbolic space ${\mathbb H}H^n$ embedded in
$SU_{2,2n}/S(U_2U_{2n})$ as a totally geodesic submanifold. Rigidity
of totally geodesic submanifolds implies that $M$ is an open part of
the tube with radius $r$ around ${\mathbb H}H^n$  in
$SU_{2,2n}/S(U_2U_{2n})$.

Now assume that $M$ satisfies property (ii) in Proposition
\ref{pccase2}. As above we define $c_p,F,X_Z,E_X$, and we get
$$
Z_X(t) = \begin{cases} E_X(t) & ,\ {\rm if}\ X \in T_\gamma,\\
\exp\left(-\frac{1}{\sqrt{2}}t\right)E_X(t) & ,\ {\rm if}\ X \in
T_\lambda,\\
\exp\left(-\sqrt{2}t\right)E_X(t) & ,\ {\rm if}\ X \in T_\alpha
\end{cases}
$$
for all $t \in {\mathbb R}$. Now consider a geodesic variation in
$SU_{2,m}/S(U_2U_m)$ consisting of geodesics $c_p$. The
corresponding Jacobi field is a linear combination of the three
types of the Jacobi fields $Z_X$ listed above, and hence its length
remains bounded when $t \to \infty$. This shows that all geodesics
$c_p$ in $SU_{2,m}/S(U_2U_m)$ are asymptotic to each other and hence
determine a singular point $z \in SU_{2,m}/S(U_2U_m)(\infty)$ at
infinity. Therefore $M$ is an integral manifold of the distribution
on $SU_{2,m}/S(U_2U_m)$ given by the orthogonal complements of the
tangent vectors of the geodesics in the asymptote class $z$. This
distribution is integrable and the maximal leaves are the
horospheres in $SU_{2,m}/S(U_2U_m)$ whose center at infinity is $z$.
Uniqueness of integral manifolds of integrable distributions finally
implies that $M$ is an open part of a horosphere in
$SU_{2,m}/S(U_2U_m)$ whose center is the singular point $z$ at
infinity.

Altogether we have now proved the following result:

\begin{thm} \label{resultcase2}
Let $M$ be a connected hypersurface in $SU_{2,m}/S(U_2U_m)$, $m \geq
2$. Assume that the maximal complex subbundle ${\mathcal C}$ of $TM$
and the maximal quaternionic subbundle ${\mathcal Q}$ of $TM$ are
both invariant under the shape operator of $M$. If $JN \perp
{\mathfrak J}N$, then one of the following statements holds:
\begin{itemize}
\item[(i)] $M$ is an open part of a tube around a totally geodesic ${\mathbb H}H^n$ in
$SU_{2,2n}/S(U_2U_{2n})$, $m = 2n$;
\item [(ii)] $M$ is an open part of a horosphere in $SU_{2,m}/S(U_2U_m)$
whose center at infinity is singular and of type $JN \perp
{\mathfrak J}N$;
\item[(iii)] $M$ has at least four distinct principal curvatures,
three of which are given by
$$
\alpha = \beta = \sqrt{2}\ ,\ \gamma = 0\ ,\ \lambda =
\frac{1}{\sqrt{2}}
$$
with corresponding principal curvature spaces
$$
T_\alpha = ({\mathcal C} \cap {\mathcal Q})^\perp\ ,\ T_\gamma =
J{\mathcal Q}^\perp\ ,\ T_\lambda \subset {\mathcal C} \cap {\mathcal Q} \cap J{\mathcal Q}.
$$
If $\mu$ is another (possibly nonconstant) principal
curvature function  , then $JT_{\mu} \subset T_\lambda$
and ${\mathfrak J}T_{\mu} \subset T_\lambda$.
\end{itemize}
\end{thm}

We investigated thoroughly case (iii), but failed to establish the existence or non-existence
of a real hypersurface having principal curvatures as described in case (iii). However,
we now conjecture that case (iii) in Theorem \ref{resultcase2} cannot
occur.

\subsection{The case $JN \in {\mathfrak J}N$} \label{case1}

In this subsection we assume that $JN \in {\mathfrak J}N$. There
exists an almost Hermitian structure $J_1 \in {\mathfrak J}$ so that
$JN = J_1N$. We then have
$$
\xi = \xi_1\ ,\ \alpha = \beta_1\ ,\ \phi\xi_2 = -\xi_3\ ,\
\phi\xi_3 = \xi_2\ ,\ \phi{\mathcal Q} \subset {\mathcal Q}\ ,\ {\mathcal Q} \subset {\mathcal C}.
$$
By inserting $X = \xi_2$ (or $X = \xi_3$) into the equation in
Proposition \ref{pcC} we get
\begin{equation} \label{eq2}
2\beta_2\beta_3 - \alpha(\beta_2+\beta_3) + 2 = 0.
\end{equation}

From Propositions \ref{pcC} and \ref{pcQ} we immediately get

\begin{lm} \label{lemmacase2}
Let $X \in {\mathcal Q}$ with $AX = \lambda X$. Then we have
\begin{eqnarray}
(2 \lambda - \alpha)A\phi X & = & (\alpha \lambda - 1) \phi X -
\phi_1X,
\label{eq11} \\
(2 \lambda - \alpha)A\phi_1 X & = & (\alpha \lambda - 1) \phi_1 X -
\phi X, \label{eq12} \\
(2\lambda - \beta_\nu)A\phi_\nu X & = & (\beta_\nu\lambda -
1)\phi_\nu X\ ,\ \nu=2,3. \label{eq13}
\end{eqnarray}
\end{lm}

By adding equations (\ref{eq11}) and (\ref{eq12}) we get
\begin{equation}
(2 \lambda - \alpha)A(\phi + \phi_1)X = (\alpha \lambda - 2) (\phi +
\phi_1)X, \label{eq14}
\end{equation}
and by subtracting (\ref{eq12}) from (\ref{eq11}) we get
\begin{equation}
(2 \lambda - \alpha)A(\phi - \phi_1)X = \alpha \lambda (\phi -
\phi_1)X. \label{eq15}
\end{equation}

Note that on ${\mathcal Q}$ we have $(\phi \phi_1)^2 = I$ and ${\rm
tr}(\phi\phi_1) = 0$. Let $E_{+1}$ and $E_{-1}$ be the eigenbundles
of $\phi \phi_1|{\mathcal Q}$ with respect to the eigenvalues $+1$ and $-1$
respectively. Then the maximal quaternionic subbundle ${\mathcal Q}$
of $TM$ decomposes orthogonally into the Whitney sum ${\mathcal Q} =
E_{+1} \oplus E_{-1}$, and the rank of both eigenbundles $E_{\pm 1}$
is equal to $2m+2$. We have $X \in E_{+1}$ if and only if $\phi X =
-\phi_1 X$ and $X \in E_{-1}$ if and only if $\phi X = \phi_1 X$.

\begin{lm} \label{2lambdaalpha}
Let $X \in {\mathcal Q}$ with $AX = \lambda X$. If $2\lambda =
\alpha$, then $\lambda = 1$, $\alpha = 2$ and $X \in E_{-1}$.
\end{lm}

\begin{proof} From ($\ref{eq14}$) and ($\ref{eq15}$) we see that
one of the following two statements holds:
\begin{itemize}
\item[(i)] $\lambda = 0$, $\alpha = 0$ and $X \in E_{+1}$;
\item[(ii)] $\lambda = 1$, $\alpha = 2$ and $X \in E_{-1}$.
\end{itemize}
In case (ii) we assume without loss of generality that $\alpha \geq
0$. We have to exclude case (i). Assume that $\lambda = 0$, $\alpha
= 0$ and $X \in E_{+1}$. From ($\ref{eq2}$) we get
$\beta_2\beta_3 = -1$, and therefore both $\beta_2$ and $\beta_3$
are nonzero. From ($\ref{eq13}$) we get
$$
A\phi_\nu X = \frac{1}{\beta_\nu}\phi_\nu X\ ,\ \nu = 2,3.
$$
By applying ($\ref{eq13}$) for $\nu = 2$ to $\phi_3 X$ we
obtain
$$ \frac{3}{\beta_3}A\phi_1 X = \left(\frac{2}{\beta_3} - \beta_2\right) A\phi_2\phi_3 X =
\left(\frac{\beta_2}{\beta_3} - 1\right) \phi_2\phi_3 X =
\left(\frac{\beta_2 - \beta_3}{\beta_3}\right)\phi_1X,
$$
and by applying ($\ref{eq13}$) for $\nu = 3$ to $\phi_2 X$
we obtain
$$ \frac{3}{\beta_2}A\phi_1 X = - \left(\frac{2}{\beta_2} - \beta_3\right) A\phi_3\phi_2 X =
- \left(\frac{\beta_3}{\beta_2} - 1\right) \phi_3\phi_2 X =
\left(\frac{\beta_3 - \beta_2}{\beta_2}\right)\phi_1X,
$$
The previous two equations imply $\beta_2 = \beta_3$, which
contradicts $\beta_2\beta_3 = -1$. It follows that case (i) cannot
hold.
\end{proof}

We denote by $\Lambda$ the set of all eigenvalues of $A|{\mathcal
Q}$, and for each $\rho \in \Lambda$ we denote by $T_\rho$ the
corresponding eigenspace.

We will first assume that there exists $\lambda \in \Lambda$ with $2
\lambda = \alpha$. Then we have $\alpha = 2$, $\lambda = 1$ and
$T_\lambda \subset E_{-1}$ according to Lemma \ref{2lambdaalpha}.
Since $\alpha = 2$, (\ref{eq2}) becomes
$$
0 = \beta_2\beta_3 - (\beta_2 + \beta_3) + 1 = (\beta_2 - 1)(\beta_3
- 1).
$$
Therefore we have $\beta_2 = 1$ or $\beta_3 = 1$. Without loss of
generality we may assume that $\beta_2 = 1$. From
(\ref{eq13}) we get $A\phi_2 X = 0$ for all $X \in T_1$. Applying
 (\ref{eq12}) to $\phi_2 X$ and using the fact that $X \in
E_{-1}$ we get $A\phi_3 X = 0$ for all $X \in T_1$. Thus we have
shown that
\begin{equation} \label{eq16}
0 \in \Lambda\ ,\ \phi_2 T_1 \subset T_0\ ,\ \phi_3 T_1 \subset T_0.
\end{equation}
Next, we apply (\ref{eq13}) for $\nu = 2$ to $\phi_3 X$,
which yields $A\phi_1 X = \phi_1 X$, and applying
(\ref{eq13}) for $\nu = 3$ to $\phi_2 X$ gives $\beta_3 A\phi_1 X =
\phi_1 X$. Comparing the previous two equations shows that $\beta_3
= 1$. Thus we have proved that
\begin{equation} \label{eq17}
\beta_2 = \beta_3 = 1\ ,\ \phi_1 T_1 \subset T_1.
\end{equation}
Now we choose $\rho \in \Lambda \setminus \{1\}$ and $Y \in T_\rho$.
From (\ref{eq16}) we know that $\Lambda \setminus \{1\} \neq
\emptyset$. From (\ref{eq13}) and (\ref{eq17}) we get $(2\rho -
1)A\phi_\nu Y = (\rho - 1)\phi_\nu Y$ for $\nu = 2,3$. Since $\rho
\neq 1$ this implies $2\rho \neq 1$ and
\begin{equation} \label{eq18}
\rho^* = \frac{\rho - 1}{2\rho -1} \in \Lambda\ ,\ \phi_2 T_\rho
\subset T_{\rho^*}\ ,\ \phi_3 T_\rho \subset T_{\rho^*}.
\end{equation}
Note that $(\rho^*)^* = \rho$ and $0^* = 1 = \lambda$. Finally, we
apply (\ref{eq13}) for $\nu = 2$ to $\phi_3 Y$ and obtain
$A \phi_1 Y = \rho \phi_1 Y$, and therefore
\begin{equation} \label{eq19}
\phi_1 T_\rho \subset T_\rho.
\end{equation}
From (\ref{eq12}) and (\ref{eq19}) we obtain $(2\rho - 2)\rho\phi_1
Y = (2\rho - 1)\phi_1 Y - \phi Y$ and therefore
$$
\phi Y = (-2\rho^2 + 4\rho - 1)\phi_1Y.
$$
Since $\phi Y$ and $\phi_1 Y$ have the same length, this implies
$$
Y \in E_{\pm 1}\ ,\ -2\rho^2 + 4\rho - 1 = \pm 1.
$$
The equation $-2\rho^2 + 4\rho - 1 = 1$ has $\rho = 1$ as a solution
with multiplicity $2$, and the equation $-2\rho^2 + 4\rho - 1 = -1$
has $\rho = 0$ and $\rho = 2$ as solutions. However, for $\rho = 2$
we would have $ \frac{1}{3} = \rho^* \in \Lambda$ according to
(\ref{eq18}), but since $\frac{1}{3}$ is not a solution of $-2\rho^2
+ 4\rho - 1 = \pm 1$, we can dismiss the case $\rho = 2$. Altogether
we have shown that $\Lambda = \{0,1\}$. From
(\ref{eq16})--(\ref{eq19}) it is clear that $T_0$ and $T_1$ have the
same dimension, and hence $\dim T_0 = \dim T_1 = 2m+2$. Since $T_1
\subset E_{-1}$ and the rank of $E_{-1}$ is $2m+2$, we get $T_1 =
E_{-1}$. From the two orthogonal decomposition ${\mathcal Q} =
E_{+1} \oplus E_{-1} = T_0 \oplus T_1$ we also get $T_0 = E_{+1}$.
Thus we have proved

\begin{prop} \label{2lambdaequalalpha}
Assume that there exists a principal curvature $\lambda \in \Lambda$
with $2 \lambda = \alpha$. Then $M$ has three distinct constant
principal curvatures $0$, $1$ and $2$ with multiplicity $2m+2$,
$2m+4$ and $1$, respectively. The corresponding principal curvature
spaces are $E_{+1}$, $E_{-1} \oplus ({\mathcal C} \ominus {\mathcal Q})$
and ${\mathcal C}^\perp$, respectively.
\end{prop}

We will now assume that $2\lambda \neq \alpha$ for all $\lambda \in
\Lambda$. The linear maps
$$
{\mathcal Q} \to E_{+1}\ ,\ X \mapsto (\phi - \phi_1)X\ \ {\rm and}\
\ {\mathcal Q} \to E_{-1}\ ,\ X \mapsto (\phi + \phi_1)X
$$
are epimorphisms, and according to (\ref{eq14}) and (\ref{eq15})
each of them maps principal curvature vectors in ${\mathcal Q}$
either to $0$ or to a principal curvature vector in $E_{+1}$ resp.\
$E_{-1}$. It follows that there exists a basis of principal
curvature vectors in ${\mathcal Q}$ such that each vector in that
basis is in $E_{+1}$ or in $E_{-1}$. In other words, we have
$$
T_\lambda = (T_\lambda \cap E_{+1}) \oplus (T_\lambda \cap E_{-1})\
{\rm for\ all}\ \lambda \in \Lambda.
$$
From (\ref{eq11}) and the $\phi$-invariance of $E_{\pm 1}$ we get

\begin{lm} \label{lemmacase2refined} Let $\lambda \in \Lambda$. Then we have
\begin{eqnarray}
A\phi X & = & \frac{\alpha \lambda}{2 \lambda - \alpha} \phi X \in
E_{+1} \
{\rm for\ all}\ X \in T_\lambda \cap E_{+1} \label{eq20} \\
A\phi X & = & \frac{\alpha \lambda - 2}{2 \lambda - \alpha} \phi X
\in E_{-1} \ {\rm for\ all}\ X \in T_\lambda \cap E_{-1}
\label{eq21}
\end{eqnarray}
\end{lm}

This shows that the cardinality $|\Lambda|$ of $\Lambda$ satisfies
$|\Lambda| \geq 2$. From (\ref{eq13}) we easily get

\begin{lm} \label{lambdabeta}
Let $\lambda \in \Lambda$. If $2\lambda = \beta_\nu$, then ($\lambda
= \frac{1}{\sqrt2}$ and $\beta_\nu = \sqrt{2}$) or ($\lambda =
-\frac{1}{\sqrt2}$ and $\beta_\nu = -\sqrt{2}$). Moreover, if
$\beta_\nu = \sqrt{2}$, then $\frac{1}{\sqrt{2}} \in \Lambda$ and
$\phi_\nu T_\lambda \subset T_{1/\sqrt{2}}$ for all $\lambda \in
\Lambda \setminus \{1/\sqrt{2}\}$. Similarly, if $\beta_\nu =
-\sqrt{2}$, then $-\frac{1}{\sqrt{2}} \in \Lambda$ and $\phi_\nu
T_\lambda \subset T_{-1/\sqrt{2}}$ for all $\lambda \in \Lambda
\setminus \{-1/\sqrt{2}\}$.
\end{lm}

\begin{lm} \label{lambda23}
Let $\lambda \in \Lambda$, $\nu \in \{2,3\}$, and assume that
$2\lambda \neq \beta_\nu$. Then we have
\begin{equation} \label{eq23}
\lambda_\nu = \frac{\beta_\nu \lambda - 1}{2\lambda - \beta_\nu} \in
\Lambda\ {\rm and}\  \phi_\nu T_\lambda \subset T_{\lambda_\nu}.
\end{equation}
Moreover, if both $2\lambda \neq \beta_2$ and $2\lambda \neq
\beta_3$, then one of the two following statements holds:
\begin{itemize}
\item[(i)] $T_\lambda \subset E_{+1}$ and $2\lambda_2\lambda_3 - \alpha(\lambda_2 + \lambda_3) +
2 = 0$;
\item[(ii)] $T_\lambda \subset E_{-1}$ and $2\lambda_2\lambda_3 - \alpha(\lambda_2 + \lambda_3) =
0$.
\end{itemize}
\end{lm}

\begin{proof}
The first statement follows from (\ref{eq13}). From (\ref{eq12}) we
obtain for $X \in T_\lambda$ that
\begin{eqnarray*} 0 & = & (2\lambda_2 - \alpha)A\phi_1\phi_2 X -
(\alpha\lambda_2 - 1)\phi_1\phi_2 X + \phi\phi_2 X \\
& = & (2\lambda_2 - \alpha)A\phi_3 X - (\alpha\lambda_2 - 1)\phi_3 X
+ \phi_2\phi X \\
& = & (2\lambda_2 - \alpha)\lambda_3 \phi_3 X - (\alpha\lambda_2 -
1)\phi_3 X + \phi_2\phi X \\
& = & (2\lambda_2\lambda_3 - \alpha(\lambda_2 + \lambda_3) +
1)\phi_3 X + \phi_2\phi X\\
& = & -(2\lambda_2\lambda_3 - \alpha(\lambda_2 + \lambda_3) +
1)\phi_2\phi_1 X + \phi_2\phi X.
\end{eqnarray*}
This implies
$$
0 = (2\lambda_2\lambda_3 - \alpha(\lambda_2 + \lambda_3) + 1)\phi_1
X - \phi X,
$$
from which the assertion easily follows.
\end{proof}

\begin{lm} \label{usefullemma}
Let $\lambda,\lambda_2,\lambda_3 \in \Lambda$ and assume that
$\phi_\nu T_\lambda \subset T_{\lambda_\nu}$ for $\nu = 2,3$. If
$2\lambda_3 \neq \beta_2$ and $2\lambda_2 \neq \beta_3$, then at
least one of the following three statements holds:
\begin{itemize}
\item[(i)] $2\lambda^2 = 1$;
\item[(ii)] $\beta_2\beta_3 = 2$;
\item[(iii)] $\beta_2 = \beta_3$.
\end{itemize}
\end{lm}

\begin{proof}
Let $X \in T_\lambda$. From Lemma \ref{lambda23} we obtain
\begin{eqnarray*}
A \phi_1 X & = & A\phi_2\phi_3 X = \frac{\beta_2\lambda_3 -
1}{2\lambda_3 - \beta_2}\phi_2\phi_3 X = \frac{(\beta_2\beta_3 -
2)\lambda + (\beta_3 - \beta_2)}{(\beta_2\beta_3 - 2) + 2(\beta_3 -
\beta_2)\lambda}\phi_1X; \\
A \phi_1 X & = & -A\phi_3\phi_2 X = -\frac{\beta_3\lambda_2 -
1}{2\lambda_2 - \beta_3}\phi_3\phi_2 X = \frac{(\beta_2\beta_3 -
2)\lambda + (\beta_2 - \beta_3)}{(\beta_2\beta_3 - 2) + 2(\beta_2 -
\beta_3)\lambda}\phi_1X.
\end{eqnarray*}
Comparing these two equations leads to $ 0 = (2\lambda^2 -
1)(\beta_2\beta_3 - 2)(\beta_3 - \beta_2)$, which implies the
assertion.
\end{proof}

\begin{lm} \label{generalcase}
If $\beta_2^2 \neq 2 \neq \beta_3^2$, then $\beta_2 = \beta_3$.
\end{lm}

\begin{proof}
If $\lambda = \frac{1}{\sqrt{2}} \in \Lambda$, then Lemma \ref{lambda23} implies
$\lambda_2 = \lambda_3 = -\frac{1}{\sqrt{2}}$ and $\alpha < 0$.
If $\lambda = - \frac{1}{\sqrt{2}} \in \Lambda$, then Lemma \ref{lambda23} implies
$\lambda_2 = \lambda_3 = \frac{1}{\sqrt{2}}$ and $\alpha > 0$.
Thus we have $\Lambda \neq \{\pm \frac{1}{\sqrt{2}}\}$, and it
follows from Lemma \ref{usefullemma} that $\beta_2\beta_3 = 2$ or
$\beta_2 = \beta_3$.

Let us assume that $\beta_2\beta_3 = 2$. From (\ref{eq2}) we obtain
$\alpha(\beta_2 + \beta_3) = 6$ and hence $\alpha \neq 0$. Moreover,
from $\beta_2\beta_3 = 2$ and (\ref{eq2}) we see that $\beta_2$ and
$\beta_3$ are the solutions of the quadratic equation
\begin{equation} \label{quadratic}
\alpha x^2 - 6x + 2\alpha = 0.
\end{equation}
From (\ref{eq23}) we obtain $\lambda_2\lambda_3 = \frac{1}{2}$. If
we choose $\lambda \in \Lambda$ with $T_\lambda \subset E_{+1}$,
Lemma \ref{lambda23} (i) implies that $\lambda_2$ and $\lambda_3$
are the solutions of the quadratic equation
$$
2\alpha x^2 - 6x + \alpha = 0.
$$
It follows that both $2\lambda_2$ and $2\lambda_3$ are solutions of
the quadratic equation (\ref{quadratic}), which means that $\beta_2
= 2\lambda_2$ or $\beta_2 = 2\lambda_3$. In both cases we deduce
$\beta_2^2 = 2$ from Lemma \ref{lambdabeta}, which is a
contradiction to the assumption. Therefore we must have
$\beta_2\beta_3 \neq 2$, and we conclude that $\beta_2 = \beta_3$.
\end{proof}

\begin{lm} \label{uniquelambda}
Assume that there exist $\lambda \in \Lambda$ and $\nu \in \{2,3\}$
such $2\lambda = \beta_\nu$. Then we have $\lambda =
\frac{1}{\sqrt{2}}$, $\beta_2 = \beta_3 = \sqrt{2}$, $\alpha =
\frac{3}{\sqrt{2}}$ and $E_{-1} \subset T_\lambda$.
\end{lm}

\begin{proof}
Without loss of generality we may assume that $2\lambda = \beta_2$.
Using Lemma \ref{lambdabeta} we can also assume that $\lambda =
\frac{1}{\sqrt{2}}$ and $\beta_2 = \sqrt{2}$ (by choosing a suitable
orientation of the normal vector). Inserting $\beta_2 = \sqrt{2}$
into (\ref{eq2}) gives $(2\sqrt{2} - \alpha)\beta_3 = \sqrt{2}\alpha
- 2$. It follows from this equation that $\alpha \neq 2\sqrt{2}$ and
\begin{equation} \label{eq22}
\beta_3 = \frac{\sqrt{2}\alpha - 2}{2\sqrt{2} - \alpha}.
\end{equation}
This implies $\beta_3 \neq -\sqrt{2}$. We first assume that $\beta_3
\neq \sqrt{2} = \beta_2$. Since $|\Lambda| \geq 2$, there exists
$\rho \in \Lambda \setminus \{\lambda\}$, and any such $\rho$
satisfies $2\rho \neq \beta_2$ (since $\rho \neq \lambda$ and
$2\lambda = \beta_2$) and $2\rho \neq \beta_3$ (since $\beta_3 \neq
\pm \sqrt{2}$ and because of Lemma \ref{lambdabeta}). From Lemma
\ref{lambda23} we see that $\phi_\nu T_\rho \subset T_{\rho_\nu}$
with $\rho_2 = \frac{1}{\sqrt{2}} = \lambda$ and $\rho_3 =
\frac{\beta_3\rho - 1}{2\rho - \beta_3} \in \Lambda$. We have
$2\rho_2 = \sqrt{2} = \beta_2 \neq \beta_3$. Therefore, if $2\rho_3
\neq \beta_2$, we deduce $\rho = -\frac{1}{\sqrt{2}}$ from Lemma
\ref{usefullemma} (since $\sqrt{2} = \beta_2 \neq \beta_3$ and $\rho
\neq \lambda$). Otherwise, if $2\rho_3 = \beta_2 = \sqrt{2}$ we get
$\frac{\beta_3\rho - 1}{2\rho - \beta_3} = \rho_3 =
\frac{1}{\sqrt{2}}$ by Lemma \ref{lambdabeta}, which is equivalent
to $(\sqrt{2}\rho + 1)\beta_3 = (\sqrt{2}\rho + 1)\sqrt{2}$. Since
$\beta_3 \neq \sqrt{2}$ this implies $\rho = -\frac{1}{\sqrt{2}}$ as
well. Altogether we conclude that $\Lambda = \{\pm
\frac{1}{\sqrt{2}}\}$. However, if $\rho = -\frac{1}{\sqrt{2}}$, we
have $1 = \frac{\beta_3\rho - 1}{2\rho - \beta_3} = \rho_3 \in
\Lambda$, which is a contradiction. Hence we must have $\beta_3 =
\beta_2 = \sqrt{2}$. From (\ref{eq22}) we then obtain $\alpha =
\frac{3}{\sqrt{2}}$. Assume that there exists $\rho \in \Lambda
\setminus \{\lambda\}$ such that $T_\rho \cap E_{-1} \neq \{0\}$,
and let $0 \neq X \in T_\rho \cap E_{-1}$. From (\ref{eq23}) we
obtain $\phi_2 X \in T_\lambda \cap E_{+1}$ and $\phi_3 X \in
T_\lambda \cap E_{+1}$. Using (\ref{eq20}) we then get
$$A\phi_3 X =  A \phi_1 \phi_2 X = - A \phi \phi_2 X =
\frac{3}{\sqrt{2}}\phi\phi_2 X = - \frac{3}{\sqrt{2}}\phi_1\phi_2 X
= - \frac{3}{\sqrt{2}}\phi_3 X,
$$
which contradicts $\phi_3 X \in T_\lambda$. Thus we conclude that
there exists no $\rho \in \Lambda \setminus \{\lambda\}$ such that
$T_\rho \cap E_{-1} \neq \{0\}$, which means that $E_{-1} \subset
T_\lambda$.
\end{proof}

We will now use the above results to derive some restrictions for
the principal curvatures:

\begin{prop} \label{pccase1}
Let $M$ be a connected hypersurface in $SU_{2,m}/S(U_2U_m)$, $m \geq
2$. Assume that the maximal complex subbundle ${\mathcal C}$ of $TM$
and the maximal quaternionic subbundle ${\mathcal Q}$ of $TM$ are
both invariant under the shape operator of $M$. If $JN \in
{\mathfrak J}N$, then one the following statements holds:
\begin{itemize}
\item[(i)] $M$ has exactly four
distinct constant principal curvatures
\begin{eqnarray*}
\alpha = 2\coth(2r)\ ,\ \beta = \coth(r)\ , \lambda_1 = \tanh(r)\ ,\
\lambda_2 = 0, & &
\end{eqnarray*}
and the corresponding principal curvature spaces are
$$
T_\alpha = {\mathcal C}^\perp\ ,\ T_\beta = {\mathcal C} \ominus {\mathcal Q}
\ ,\ T_{\lambda_1} = E_{-1}\ ,\ T_{\lambda_2} = E_{+1}.
$$
The principal curvature spaces $T_{\lambda_1}$ and $T_{\lambda_2}$
are complex (with respect to $J$) and totally complex (with respect
to ${\mathfrak J}$).
\item[(ii)] $M$ has exactly three distinct constant principal curvatures
$$
\alpha = 2\ ,\ \beta = 1\ ,\ \lambda = 0
$$
with corresponding principal curvature spaces
$$
T_\alpha = {\mathcal C}^\perp\ ,\ T_\beta = ({\mathcal C} \ominus {\mathcal Q})
\oplus E_{-1}\ ,\ T_\lambda = E_{+1}.
$$
\item[(iii)] We have $\alpha = \frac{3}{\sqrt{2}}$, $\beta_2 =
\beta_3 = \sqrt{2}$, $|\Lambda| \geq 2$, $\lambda =
\frac{1}{\sqrt{2}} \in \Lambda$ and $E_{-1} \subset T_\lambda$.
\end{itemize}
\end{prop}

\begin{proof}
If there exists a principal curvature $\lambda \in \Lambda$ such
that $2\lambda = \beta_\nu$ for some $\nu \in \{2,3\}$, we get
statement (iii) from Lemma \ref{uniquelambda}. We now assume that
$2\lambda \neq \beta_\nu$ for all $\lambda \in \Lambda$ and all $\nu
\in \{2,3\}$. We have to show that $M$ satisfies (i), (ii) or (iv).

If there exists a principal curvature $\lambda \in \Lambda$ with
$2\lambda = \alpha$, we get case (ii) from Proposition
\ref{2lambdaequalalpha}. Thus we can assume that $2\lambda \neq
\alpha$ for all $\lambda \in \Lambda$. Since $2\lambda \neq
\beta_\nu$ for all $\lambda \in \Lambda$, we obtain from Lemma
\ref{lambdabeta} that $\beta_\nu^2 \neq 2$ for $\nu \in \{2,3\}$.
From Lemma \ref{generalcase} we obtain
$\beta_2 = \beta_3 =: \beta$, and (\ref{eq2}) implies
\begin{equation} \label{eq2beta2beta3}
\beta^2 - \alpha\beta + 1 = 0.
\end{equation}
Note that if $\beta$ is a solution of
(\ref{eq2beta2beta3}), then $\frac{1}{\beta}$ is the other solution.
In case $\beta = 1$ (\ref{eq2beta2beta3}) has a root of multiplicity
two. Let $\lambda \in \Lambda$ and $X \in T_\lambda$. Then, using
Lemma \ref{lambda23}, we see that for $\nu \in \{2,3\}$ we have
$A\phi_\nu X = \lambda^*\phi_\nu X$ with
$$ \lambda^* = \frac{\beta\lambda - 1}{2\lambda - \beta}.$$
Note that $(\lambda^*)^* = \lambda$ since $\beta^2 \neq 2$.
Moreover, we have $T_\lambda \subset E_{\pm 1}$ and
\begin{eqnarray}
T_{\lambda^*} \subset E_{-1}\ {\rm and} & \lambda^{*2} -
\alpha\lambda^* + 1 = 0 & {\rm if}\ T_\lambda \subset E_{+1} \label{eq30}\\
T_{\lambda^*} \subset E_{+1}\ {\rm and} & \lambda^*(\lambda^* -
\alpha) = 0 & {\rm if}\ T_\lambda \subset E_{-1} \label{eq31}
\end{eqnarray}
We choose $\lambda \in \Lambda$ such that $T_\lambda \subset
E_{-1}$. From (\ref{eq2beta2beta3}) and (\ref{eq31}) we then obtain
$\lambda^* = 0$ or $\lambda^* = \alpha$. Assume that $\lambda^* =
\alpha$. Then we have $\alpha \in \Lambda$ and $T_\alpha \subset
E_{+1}$. From (\ref{eq2beta2beta3}) and (\ref{eq30}) we then get
$\alpha^* \in \{\beta,\frac{1}{\beta}\}$. The equation $\alpha^* =
\beta$ is equivalent to $\beta^2 - \alpha\beta - 1 = 0$, which
contradicts (\ref{eq2beta2beta3}). The equation $\alpha^* =
\frac{1}{\beta}$ is equivalent to $\alpha(\beta^2 - 2) = 0$. Since
$\beta^2 \neq 2$ this implies $\alpha = 0$, which contradicts
(\ref{eq2beta2beta3}). Therefore $\lambda^* = \alpha$ is impossible,
and we conclude that $\lambda^* = 0$ and hence $\lambda =
\frac{1}{\beta}$. Altogether we conclude that $\Lambda =
\{0,\frac{1}{\beta}\}$, $T_0 = E_{+1}$ and $T_{\frac{1}{\beta}} =
E_{-1}$.

From Equation (\ref{gradient}) we see that the gradient ${\rm
grad}^\alpha$ of $\alpha$ on $M$ satisfies ${\rm grad}^\alpha =
(\xi\alpha)\xi$, and as in the proof of Proposition \ref{pccase2} we
obtain $(\xi\alpha)g((A \phi + \phi A)X ,Y) = 0$ for all vector
fields $X,Y$ on $M$. Since $E_{-1}$ is invariant under both $\phi$
and $A$, we get for all $X \in E_{-1}$ that
$$
0 = (\xi\alpha)g((A \phi + \phi A)X ,\phi X) =
\frac{1}{\beta}(\xi\alpha)g(\phi X,\phi X),
$$
which implies $\xi\alpha = 0$. Therefore ${\rm grad}^\alpha =
(\xi\alpha)\xi= 0$, and since $M$ is connected we conclude that
$\alpha$ is constant.

It follows easily from (\ref{eq2beta2beta3}) that $\alpha^2 \geq 4$.
If $\alpha^2 = 4$, then $\beta = \frac{\alpha}{2}$ and hence
$2\lambda_1 - \alpha = 0$ for $\lambda_1 = \frac{1}{\beta} \in
\Lambda$, which contradicts our assumption that $2\lambda \neq
\alpha$ for all $\lambda \in \Lambda$. Thus we have $\alpha^2 > 4$
and we can write $\alpha = 2\coth(2r)$ for some $r \in {\mathbb
R}_+$ (and possibly changing the orientation of the normal vector).
From (\ref{eq2beta2beta3}) we then obtain $\beta = \coth(r)$ and
therefore $\lambda_1 = \frac{1}{\beta} = \tanh(r)$. Altogether this
shows that statement (i) holds.
\end{proof}

Assume that $M$ satisfies property (i) in Proposition \ref{pccase1}.
For $p \in M$ we denote by $c_p : {\mathbb R} \to
SU_{2,m}/S(U_2U_m)$ the geodesic in $SU_{2,m}/S(U_2U_m)$ with
$c_p(0) = p$ and $\dot{c}_p(0) = N_p$, and define the smooth map
$$
F : M \to SU_{2,m}/S(U_2U_m), p \mapsto c_p(r).
$$
Geometrically, $F$ is the displacement of $M$ at distance $r$ in
direction of the unit normal vector field $N$. For each $p \in M$
the differential $d_pF$ of $F$ at $p$ can be computed using Jacobi
vector fields by means of $d_pF(X) = Z_X(r)$, where $Z_X$ is the
Jacobi vector field along $c_p$ with initial value $Z_X(0) = X$ and
$Z_X^\prime(0) = -AX$. Using the explicit description of the Jacobi
operator $R_N$ given in Table \ref{Jacobi} for the case $JN = J_1N
\in {\mathfrak J}N$ we get
$$
Z_X(r) = \begin{cases} E_X(r) & ,\ {\rm if}\ X \in T_{\lambda_2},\\
(\cosh(r) - \kappa\sinh(r))E_X(r) & ,\ {\rm if}\ X \in T_\kappa\
{\rm and}\ \kappa \in
\{\beta,\lambda_1\},\\
\left(\cosh\left(2r\right) -
\frac{\alpha}{2}\sinh\left(2r\right)\right)E_X(r) & ,\ {\rm if}\ X
\in T_\alpha,
\end{cases}
$$
where $E_X$ denotes the parallel vector field along $c_p$ with
$E_X(0) = X$. This shows that the kernel ${\rm ker}\,dF$ of $dF$ is
given by
$$
{\rm ker}\,dF = T_\alpha \oplus T_\beta = {\mathfrak J}N = {\mathcal Q}^\perp,
$$
and that $F$ is of constant rank equal to the rank of the
quaternionic vector bundle ${\mathcal Q}$, which is equal to
$4(m-1)$. Thus, locally, $F$ is a submersion onto a
$4(m-1)$-dimensional submanifold $B$ of $SU_{2,m}/S(U_2U_m)$.
Moreover, the tangent space of $B$ at $F(p)$ is obtained by parallel
translation of $(T_{\lambda_1} \oplus T_{\lambda_2})(p) = {\mathcal
Q}(p)$, which is a quaternionic (with respect to ${\mathfrak J}$)
and complex (with respect to $J$) subspace of
$T_pSU_{2,m}/S(U_2U_m)$. Since both $J$ and ${\mathfrak J}$ are
parallel along $c_p$, also $T_{F(p)}B$ is a quaternionic (with
respect to ${\mathfrak J}$) and complex (with respect to $J$)
subspace of $T_{F(p)}SU_{2,m}/S(U_2U_m)$. Thus $B$ is a quaternionic
and complex submanifold of $SU_{2,m}/S(U_2U_m)$. Since every
quaternionic submanifold of a quaternionic K\"{a}hler manifold is
necessarily totally geodesic (see e.g.\ \cite{G}), we see that $B$
is a totally geodesic submanifold of $SU_{2,m}/S(U_2U_m)$.
Using the concept of duality between the symmetric spaces
$SU_{2,m}/S(U_2U_m)$ and $SU_{2+m}/S(U_2U_m)$, it follows from the
classification of totally geodesic submanifolds in complex
$2$-plane Grassmannians (see \cite{K}), that $B$ is an open part of
$SU_{2,m-1}/S(U_2U_{m-1})$ embedded in $SU_{2,m}/S(U_2U_m)$ as a
totally geodesic submanifold. Rigidity of totally geodesic
submanifolds implies that $M$ is an open part of the tube with
radius $r$ around $SU_{2,m-1}/S(U_2U_{m-1})$  in
$SU_{2,m}/S(U_2U_m)$.

Now assume that $M$ satisfies property (ii) in Proposition
\ref{pccase1}. As above we define $c_p,F,X_Z,E_X$, and we get
$$
Z_X(t) = \begin{cases} E_X(t) & ,\ {\rm if}\ X \in T_\lambda,\\
\exp(-t)E_X(t) & ,\ {\rm if}\ X \in
T_\beta,\\
\exp(-2t)E_X(t) & ,\ {\rm if}\ X \in T_\alpha
\end{cases}
$$
for all $t \in {\mathbb R}$. Now consider a geodesic variation in
$SU_{2,m}/S(U_2U_m)$ consisting of geodesics $c_p$. The
corresponding Jacobi field is a linear combination of the three
types of the Jacobi fields $Z_X$ listed above, and hence its length
remains bounded when $t \to \infty$. This shows that all geodesics
$c_p$ in $SU_{2,m}/S(U_2U_m)$ are asymptotic to each other and hence
determine a singular point $z \in SU_{2,m}/S(U_2U_m)(\infty)$ at
infinity. Therefore $M$ is an integral manifold of the distribution
on $SU_{2,m}/S(U_2U_m)$ given by the orthogonal complements of the
tangent vectors of the geodesics in the asymptote class $z$. This
distribution is integrable and the maximal leaves are the
horospheres in $SU_{2,m}/S(U_2U_m)$ whose center at infinity is $z$.
Uniqueness of integral manifolds of integrable distributions finally
implies that $M$ is a open part of a horosphere in
$SU_{2,m}/S(U_2U_m)$ whose center is the singular point $z$ at
infinity.

Finally, assume that $M$ satisfies property (iii) in Proposition \ref{pccase1}.
Let $t \in {\mathbb R}_+$ such that $\coth(t) =
\sqrt{2} = \beta$. Then we have $\alpha = \frac{3}{\sqrt{2}} = 2\coth(2t)$ and
$\lambda = \frac{1}{\sqrt{2}} = \tanh(t)$.
As above we define $c_p,F,X_Z,E_X$.
Since $M$ is a hypersurface, also $M_r =
F(M)$ is (locally) a hypersurface for sufficiently small $r \in {\mathbb R}_+$.
The tangent vector $\dot{c}_p(r)$ is a unit normal vector of $M_r$ at $c_p(r)$.
Since $\dot{c}_p(0) = N_p$ is a singular tangent vector of type $JX \in {\mathcal J}X$,
every tangent vector of $c_p$ is singular and of type $JX \in {\mathcal J}X$.
The tangent space $T_{c_p(r)}M_r$ of $M_r$ at $c_p(r)$ is obtained by
parallel translation of $T_pM$ along $c_p$ from $c_p(0)$ to $c_p(r)$.
We denote by ${\mathcal C}_r$ the maximal complex subbundle of $TM_r$ and by
${\mathcal Q}_r$ the maximal quaternionic subbundle of $TM_r$.
Let $A_r$ be the shape operator of $M_r$ with respect to $\dot{c}_p(r)$.
For $X \in T_\alpha$ we have $A_rZ_X(r) = -Z_X^\prime(r)$ with
$$
Z_X(r) = \left(\cosh(2r) - \coth(2t)\sinh(2r)\right)E_X(r).
$$
It follows that $E_X(r)$ is a principal curvature vector of $M_r$ with corresponding
principal curvature
$$
\alpha_r = -\frac{2\sinh(2r) - 2\coth(2t)\cosh(2r)}{\cosh(2r) - \coth(2t)\sinh(2r)} = 2\coth(2(r+t)).
$$
Since $T_\alpha = {\mathcal C}^\perp = {\mathbb R}JN$ and $J$ is parallel, we
see that $({\mathcal C}_r)^\perp$, and hence also ${\mathcal C}_r$, are invariant under the shape
operator of $M_r$. For $X \in T_\beta$ we have $A_rZ_X(r) = -Z_X^\prime(r)$ with
$$
Z_X(r) = \left(\cosh(r) - \coth(t)\sinh(r)\right)E_X(r).
$$
Since $T_\beta = {\mathcal C} \ominus {\mathcal Q}$ and both $J$ and ${\mathcal J}$
are parallel, we conclude that ${\mathcal C}_r \ominus {\mathcal Q}_r$ is invariant under $A_r$.
Altogether this implies that ${\mathcal Q}_r$ is invariant under the shape operator of $M_r$.
We thus have proved that $M_r$ satisfies the assumptions of Proposition \ref{pccase1}.
It is easy to see that $\alpha_r \notin \{\frac{3}{\sqrt{2}},2\}$ for $r \in {\mathbb R}_+$, and hence
the principal curvatures of $M_r$ must satisfy (i) in Proposition \ref{pccase1}.
Therefore $M_r$ is an open part of a tube with radius $r+t$ around a totally geodesic $SU_{2,m-1}/S(U_2U_{m-1})$ in
$SU_{2,m}/S(U_2U_m)$. This implies that $M$ is an open part of a tube with radius $t$ around
a totally geodesic $SU_{2,m-1}/S(U_2U_{m-1})$ in $SU_{2,m}/S(U_2U_m)$.

Altogether we have now proved the following result:

\begin{thm} \label{resultcase1}
Let $M$ be a connected hypersurface in $SU_{2,m}/S(U_2U_m)$, $m \geq
2$. Assume that the maximal complex subbundle ${\mathcal C}$ of $TM$
and the maximal quaternionic subbundle ${\mathcal Q}$ of $TM$ are
both invariant under the shape operator of $M$. If the normal bundle of $M$
consists of singular tangent vectors of type $JX \in
{\mathfrak J}X$, then one the following statements holds:
\begin{itemize}
\item[(i)] $M$ is an open part of a tube around a totally geodesic $SU_{2,m-1}/S(U_2U_{m-1})$ in
$SU_{2,m}/S(U_2U_m)$;
\item [(ii)] $M$ is an open part of a horosphere in $SU_{2,m}/S(U_2U_m)$
whose center at infinity is singular and of type $JX \in {\mathfrak
J}X$.
\end{itemize}
\end{thm}

The main result, Theorem \ref{mainresult}, now follows by combining Proposition \ref{key},
Theorem \ref{resultcase2} and Theorem \ref{resultcase1}.

\end{document}